\documentclass[11pt]{article} 

\usepackage[utf8]{inputenc} 

\usepackage{cancel}

\usepackage{geometry} 
\usepackage{graphicx} 
\usepackage[english]{babel}

\usepackage{booktabs} 
\usepackage{array} 
\usepackage{paralist} 
\usepackage{verbatim} 
\usepackage{subfig} 
\usepackage{amssymb}
\usepackage{bbm}

\usepackage{amsmath}
\usepackage{amsthm}
\usepackage{makecell}

\usepackage{epigraph}

\usepackage{hyperref}
\hypersetup{
    colorlinks=true,
    citecolor=blue
}

\usepackage{pdflscape}
\usepackage{natbib}
\usepackage{algorithm2e}

\usepackage{wrapfig}
\usepackage{subfig}
\usepackage[capitalize,noabbrev]{cleveref}
\usepackage{authblk}
\usepackage{changepage}

\usepackage{tikz-cd}
\usepackage{xcolor}

\renewcommand{\le}{\leqslant}
\renewcommand{\leq}{\leqslant}

\renewcommand{\ge}{\geqslant}

\newcommand{\id}{\mathbbmss 1}

\newcommand {\matl}{\left[ \begin{matrix}}
\newcommand {\matr}{\end{matrix}\right]}
\newcommand {\Exp}{ \mathbb E }

\renewcommand {\Pr}{ \mathbb P }

\newcommand {\Var}{\mathbf{Var}}

\newcommand{\kl}{D_{\operatorname{KL}}}

\renewcommand{\d}{\mathrm{d}}

\newcommand{\cA}{\mathcal{A}}

\newcommand{\cF}{\mathcal{F}}

\newcommand{\cS}{\mathcal{S}}

\newcommand{\cB}{\mathcal{B}}

\newcommand{\cM}{\mathcal{M}}
\newcommand{\cH}{\mathcal{H}}
\newcommand{\cP}{\mathcal{P}}
\newcommand{\cU}{\mathcal{U}}
\newcommand{\cD}{\mathcal{D}}
\newcommand{\cV}{\mathcal{V}}

\newcommand{\fF}{\mathbf{F}}
\newcommand{\fA}{\mathbf{A}}
\newcommand{\fB}{\mathbf{B}}
\newcommand{\fQ}{\mathbf{Q}}
\newcommand{\fH}{\mathbf{H}}
\newcommand{\fG}{\mathbf{G}}
\newcommand{\fM}{\mathbf{M}}
\newcommand{\fR}{\mathbf{R}}
\newcommand{\fD}{\mathbf{D}}
\newcommand{\fU}{\mathbf{U}}

\newcommand{\CI}{\operatorname{CI}}

\newcommand{\iid}{\overset{\mathrm{iid}}{\sim}}

\DeclareMathAlphabet{\mathbbmsl}{U}{bbm}{m}{sl}

\DeclareMathOperator*{\Expw}{\Exp}
\DeclareMathOperator*{\Prw}{\Pr}
\DeclareMathOperator*{\Varw}{\Var}

\DeclareMathOperator*{\argmax}{arg\,max}
\DeclareMathOperator*{\argmin}{arg\,min}

\newcommand{\hwcmt}[1]{}

\newcommand{\eff}{{\mathrm{eff}}}
\newcommand{\efn}[1]{ {^\circ({#1}^\circ)} }
\newcommand{\iids}{{\mathrm{iid}}}
\newcommand{\ber}{\operatorname{ber}}

\newtheorem{theorem}{Theorem}[section]
\newtheorem{definition}[theorem]{Definition}

\newtheorem{fact}{Fact}
\newtheorem{proposition}[theorem]{Proposition}

\newtheorem{corollary}[theorem]{Corollary}
\newtheorem{lemma}[theorem]{Lemma}
\theoremstyle{definition}\newtheorem{remark}[theorem]{Remark}

\Crefname{fact}{Fact}{Facts}
\Crefname{assumption}{Assumption}{Assumptions}

\usepackage{tocbasic}
\DeclareTOCStyleEntry[
  beforeskip=.05em plus .6pt,
  pagenumberformat=\textbf
]{tocline}{section}

\title{E-values and sequential power-one tests \\ for monotonicity and unimodality}

\author[1]{Hongjian Wang}
\author[2]{Aaditya Ramdas}
\affil[1, 2]{Department of Statistics and Data Science, Carnegie Mellon University}
\affil[2]{Machine Learning Department, Carnegie Mellon University} 
\affil[ ]{\texttt{ \{hjnwang,aramdas\}@cmu.edu  }}
\date{\today}

\begin{document}

\maketitle

\begin{abstract}
We develop e-values and e-processes testing the null hypothesis that a distribution over nonnegative integers is monotone, and that a distribution over integers is unimodal given a certain mode. Our e-processes lead to tests of power one under any non-null distribution with a sequence of i.i.d.\ observations, and consistent set-valued mode estimators that eventually equal the true set of modes. Additionally, we characterize the set of all e-values, and therefore the set of all valid tests, with one monotone and unimodal observation, as well as the most powerful e-value for a fixed alternative.
    We then show that many of our results can be generalized to continuous random variables, relating them to the existing results in the shape-constrained inference literature.
\end{abstract}

{ \vfill \setcounter{tocdepth}{1}
    \hypersetup{linkcolor=black} \tableofcontents}
    \vfill
\newpage

\section{Introduction} \label{sec:intro}

Power-one sequential nonparametric tests with uniformly small type I error probability are of major interest in the modern statistics literature. These are
sequential hypothesis testing procedures for a nonparametric class of null distributions $\cP_0$ such that
\begin{equation}
    \Prw_{(X_n)\iid P}( \text{reject after finite observations} ) \le \alpha,\quad \forall P \in \cP_0,
\end{equation}
and
\begin{equation}
    \Prw_{(X_n)\iid Q}(\text{reject after finite observations}) = 1,\quad \forall Q \in \cP\setminus\cP_0,
\end{equation}
where the ambient class $\cP$ contains all distributions allowed in the statistical model.
Many important nonparametric pairs of classes $\cP \subseteq \cP_0$ have been studied in the past literature. To name a few landmark results, \cite{darling1968some,shekhar2023nonparametric} considered the two-sample comparison classes; a long list of authors including \cite{robbins1970statistical,shafer2005probability,howard2020time,WANG2023cat,waudbysmith2024betting} studied the classes for one- or two-sided \emph{mean} testing under a plethora of nonparametric regularity assumptions (e.g.\ subGaussian, sub-$\psi$, bounded $p$-th moment for $p \in (1,2]$, boundedness); \cite{howard2022sequential} studied the classes for \emph{median} testing (more generally, testing any quantile) without any distribution assumption. In this paper, we study the classes for a location testing problem not covered in the literature: \emph{mode} testing.

Many of these past works, either explicitly or implicitly, are based on the concept of \emph{e-values} and \emph{e-processes}. An e-value for testing the null hypothesis ``$P \in \cP_0$'' (or simply, ``an e-value for $\cP_0$'') is defined as a nonnegative statistic $E$ such that $\Expw_{P}E \le 1$ under any null distribution $P \in \cP_0$. An e-process for $\cP_0$ is then defined as a sequence of statistics $M_n = M(X_1,\dots, X_n)$ such that $M_\tau$ yields an e-value for $\cP_0$ at any stopping time $\tau$ on the data filtration. A sequential test controlling type I error rate within $\alpha$ is then formed by rejecting the null ``$P \in \cP_0$'' whenever $(M_n)$ first surpasses $1/\alpha$ \citep[Chapter 7]{ramdas2024hypothesis}. For this test to be power-one, $(M_n)$ must grow past $1/\alpha$ under any alternative distribution $Q \in \cP\setminus\cP_0$. In fact, authors (including us in this paper) often show that $M_n \to \infty$ under any such $Q$, usually exponentially fast. Our main results in this paper concern the construction of
\begin{itemize}
    \item \underline{(e-processes for discrete monotonicity)} e-processes for $\cM$, the class of distributions on nonnegative integers $\mathbb Z^{\ge 0}$ with monotone non-increasing probability mass function, that grow to $\infty$ exponentially fast under any distribution on $\mathbb Z^{\ge 0}$ not in $\cM$;
    \item \underline{(e-processes for discrete $\theta$-unimodality)} e-processes for $\cD_\theta$, the class of distributions on integers $\mathbb Z$ with monotone non-increasing probability mass function on $\{ \theta, \theta+1 ,\dots \}$ and non-decreasing probability mass function on $\{ \dots, \theta -1, \theta \}$, that grow to $\infty$ exponentially fast under any distribution on $\mathbb Z$ not in $\cD_\theta$.
\end{itemize}
The second result above is for a fixed mode $\theta \in \mathbb Z$, but it implies a power-one sequential test (though not exponentially growing e-processes) for discrete unimodality with \emph{unrestricted} mode; that is, for $\bigcup_{\theta \in \mathbb Z} \cD_{\theta}$.

A recent paper by \cite{ram2026power} resolves the problem of constructing exponentially powerful nonparametric e-processes (and thus power-one sequential tests) with an elegant topological condition: that $\cP_0$ is \emph{weakly compact} (i.e.\ compact in the topology of weak convergence of measures) within $\cP$, the set of all distributions on a given space. It is worth noting that our problems of monotonicity and unimodality testing do not satisfy this assumption, as neither $\cM$ nor $\cD_\theta$ is weakly compact. To see this, $\cM$ is not a tight family of measures since it contains the uniform distribution on $\{ 0, \dots, n \}$ for arbitrary $n$. By Prokhorov's theorem (see e.g.\ \cite{billingsley2013convergence}), $\cM$ is therefore not relatively compact, and therefore not compact, in the topology of weak convergence. The argument for $\cD_{\theta}$ is analogous.
This renders the problems we study non-trivial in light of the results of \cite{ram2026power}.

Additionally, as is articulated in the monograph by \cite{ramdas2024hypothesis}, the construction and characterization of e-values for a given distribution class is a task of independent interest in itself and might enable various downstream applications. We thus offer in this paper a near-complete study of the problem ``e-values for monotonicity and unimodality''. Beyond (and related to) the main results of power-one sequential tests with e-processes mentioned above, we also study: (1) the set of \emph{all} e-values (therefore all valid tests) for $\cM$ and $\cD_{\theta}$ given \emph{one} observation; (2) the optimal e-values and e-processes for $\cM$ and $\cD_{\theta}$ for a \emph{fixed} alternative given both one and multiple (possibly randomly many) observations; and (3) generalization of these results to continuous monotone and unimodal distributions. Many of these \emph{hypothesis testing} results also have \emph{estimation} consequences, leading to confidence sets and estimators for the unknown mode(s) that we shall spell out during the course.

A few other interesting papers in the past relevant to our study are as follows. \cite{Edelman01111990} constructs a one-observation confidence interval (OOCI) for the mode with one continuous unimodal observation, inspired by prior OOCIs for one Gaussian \citep{abbott1962two,Portnoy02012019} observation. Subsequent works inspired by this unimodal OOCI include a multivariate generalization by \citet[Theorem 5.1]{Andel}, and CIs from multiple possibly dependent observations by \citet[Theorems 7 and 10]{paul2025finite}. We will in particular relate some of our one-observation findins to those by \cite{Edelman01111990}. In a different line of work, monotonicity and unimodality frequently arise in shape-constrained estimation and inference, and the methodology often involves employing the left derivative of the least concave majorant (LCM) of the empirical cumulative distribution function as the nonparametric maximum likelihood estimator
\citep{grenander1956theory,grenander1981abstract,rao1969estimation,birge1989grenader,samworth2025nonparametric}. While the objective of these papers is usually to estimate the entire monotone or unimodal (with fixed mode) ground-truth (i.e.\ estimate $P \in \cM$ or $\cD_{\theta}$ with some $\hat P_n$), we pursue the orthogonal objectives in this work of constructing tests (i.e.\ decide \emph{if} $P \in \cM$ or $\cD_{\theta}$) and mode confidence sets (i.e.\ for what $\theta$ does $P$ does $P$ belong to $\cD_\theta$?). Nonetheless, as we will soon discuss, we will adopt tools and ideas from the nonparametric estimation literature when we consider optimal tests.

\paragraph{Notations.} 
Throughout the paper, we use $\cP(\mathbb A)$ to denote the set of probability measures over the underlying space $\mathbb A$, where $\mathbb A$ is $\mathbb Z$, $\mathbb Z^{\ge 0} = \{ 0,1,2,\dots \}$, $\mathbb R$, or $\mathbb R^{\ge 0}$. If $P \in \cP(\mathbb Z)$, we denote by $f_{P}$ the probability mass function (PMF) of $P$, i.e.\ $f_{P}(n) = P(\{n\})$. 
The set of nonnegative measurable functions $\mathbb A \to [0,\infty)$ is denoted by $\fF(\mathbb A)$. An additional suite of notations will be introduced later in \cref{sec:1obs}.

\paragraph{Paper organization.}
This paper is divided into three major parts. In \cref{sec:seq}, we develop power-one tests based on ``mixture of witness'' e-processes for discrete monotonicity and $\theta$-unimodality. To complement these sequential results, we present in \cref{sec:1obs} a full-on one-observation study of the discrete monotone and $\theta$-unimodal classes via convex duality lens, characterizing the set of all e-values, hence all tests, for these null hypotheses; as well as of the information-theoretic optimality criteria of these problems. Finally, we explore some generalizations into continuous distributions that bear closer connection to the past literature in \cref{sec:cont}. 

Some auxiliary discussions and proofs are in \cref{sec:misc,sec:pf}.

\section{E-processes for monotonicity and unimodality on $\mathbb Z$} \label{sec:seq}

Recall that we denote by $\cM$ the set of all monotone distributions on $\mathbb Z^{\ge 0}$:
\begin{equation}
    \cM = \left\{ P \in \cP(\mathbb Z^{\ge 0}) : f_P(n) \ge   f_P (n+1) \text{ for all }n \in  \mathbb Z^{\ge 0}   \right\},
\end{equation}
and by $\cD_{\theta}$ the set of all $\theta$-unimodal distributions on $\mathbb Z$:
\begin{equation}
    \cD_{\theta} = \left\{ P \in \cP(\mathbb Z) :  \text{$f_P(n-1) \le f_P(n)$ if $n \le \theta$ and $f_P(n) \le f_P(n+1)$ if $n \ge \theta$}  \right\}.
\end{equation}
It is convenient for us now to define the sets of all (one-observation) e-values for $\cM$ and $\cD_\theta$:
\begin{gather}
    \fM = \{ E \in \fF(\mathbb Z^{\ge 0}) :  \Expw_{X \sim P} E(X) \le 1 \text{ for all } P \in \cM \},
\\
    \fD_\theta = \{  E \in \fF(\mathbb Z) :  \Expw_{X \sim P} E(X) \le 1 \text{ for all } P \in \cD_\theta  \}. 
\end{gather}
In fact, we shall develop mathematical tools to describe the sets $\fM$ and $\fD_\theta$ with real parameters later in \cref{sec:1obs}. For the current section, it suffices for us to keep these two notations handy and verify that a few simple functions belong to them, as they will act as the building blocks for
the exponentially powerful e-processes we would like to construct.

\subsection{Power-one e-process for $\cM$ against $\cM^c$}

We begin the construction with the following lemma: given a \emph{fixed} alternative distribution $Q \in \cP(\mathbb Z^{\ge 0}) \setminus \cM$, we can always (with oracle access to the PMF of $Q$) construct a simple ``wavelet'' e-value $E \in \fM$ with positive \emph{e-power} $\Exp_{X \sim Q} \log E(X) > 0$.
The e-power $\Exp_{X \sim Q} \log E(X)$ is a crucial power metric for e-values (see e.g.\ Section 2 in \cite{vovk_wang_2024} and Chapter 3 in \cite{ramdas2024hypothesis}) that originates from the \cite{kelly1956new} criterion, and the condition that the e-power is positive $\Exp_{X \sim Q} \log E(X) > 0$ is stronger than the condition that $E$ is \emph{not} an e-value under this alternative distribution $\Exp_{X \sim Q} E(X) > 1$ by Jensen's inequality. We say that the e-value $E$  ``witnesses'' the alternative $Q$.

\begin{lemma}\label[lemma]{lem:log-power-monotone-evalue}
    Let $Q \in \cP(\mathbb Z^{\ge 0}) \setminus \cM$. Suppose $f_Q(m+1) - f_Q(m) > 0$. Then, the function $E^{Q,m} \in \fF(\mathbb Z^{\ge 0})$ defined as
    \begin{equation}\label{eqn:def-EQm}
    E^{Q,m}(x) = 1 + \lambda_{Q,m}( - \id_{\{ x=m \}} + \id_{\{ x=m + 1 \}})
\end{equation}
where
\begin{equation}\label{eqn:choice-of-lambda}
    \lambda_{Q,m} = \frac{f_Q(m+1) - f_Q(m)}{2( f_Q(m)+ f_Q(m+1)  )} 
\end{equation}
satisfies 
\begin{equation}\label{eqn:e-power-block}
   E^{Q,m} \in \fM  \quad \text{and} \quad \Expw_{X \sim Q} \log E^{Q,m}(X) \ge \frac{(f_Q(m+1) - f_Q(m))^2}{2( f_Q(m)+ f_Q(m+1)  )  }  > 0.
\end{equation}
\begin{proof}
To verify that $E^{Q,m} \in \fM$, it is clear that $| \lambda_{Q,m}| < 1$ which implies $E^{Q,m} > 0$, and that $\Expw_{X \sim P}  E^{Q,m}(X) = 1 + \lambda_{Q,m}(f_P(m+1) - f_P(m)) \le 1 $ under any $P \in \cM$.

Next, note that $\log(1+x) \ge x - x^2$ for $|x| \le 1/2$ and $ \lambda_{Q,m} \in (0,1/2)$. We have
    \begin{align}
         & \Expw_{X \sim Q} \log E^{Q,m}(X) = f_Q(m) \log(1 -  \lambda_{Q,m} ) +   f_Q(m+1) \log(1 +  \lambda_{Q,m} )
         \\
         \ge & f_Q(m)(  -  \lambda_{Q,m} - \lambda_{Q,m}^2 ) +  f_Q(m+1)  (   \lambda_{Q,m} - \lambda_{Q,m}^2 ) = \frac{(f_Q(m+1) - f_Q(m))^2}{2( f_Q(m)+ f_Q(m+1)  )  } .  \label{eqn:e-power-quadratic}
    \end{align}
    This concludes the proof.
\end{proof}
\end{lemma}

As can be seen from the proof, the ``wavelet amplitude'' $\lambda_{Q,m}$ \eqref{eqn:choice-of-lambda} is chosen to maximize the quadratic in \eqref{eqn:e-power-quadratic}.
It is worth noting that given a fixed alternative ground truth $Q$, there exist e-values for $\cM$ with larger e-power than $E^{Q,m}$, which
we discuss in \cref{sec:ripr}. The e-values defined above are much more convenient for our objective of constructing a power-one test.

One-observation wavelet e-values $E^{Q,m}$ compose the following power-one e-process when the statistician knows by oracle the location $m$ where the monotonicity constraint $f_Q(m) \ge f_Q(m+1)$ is violated by the data-generating distribution $Q$, via the technique commonly referred to as ``predictable plug-in'' in the e-value literature: at each sample size $k$, a predictable choice (i.e.\ based on $X_1,\dots, X_{k-1}$) of e-value is multiplied to form the e-process. In the rest of this paper, we denote the empirical measure
\begin{equation}
    \hat Q_n = \frac{1}{n}\sum_{k=1}^n \delta_{X_k},
\end{equation}
where $\delta_x$ is the Dirac point mass on $x$.

\begin{lemma}\label[lemma]{lem:oracle-power1}
     Let $Q \in \cP(\mathbb Z^{\ge 0}) \setminus \cM$. Suppose $f_Q(m+1) - f_Q(m) = \delta > 0$. Then, the process
     \begin{equation}
         M_n^{(m)} = \prod_{k=1}^n E^{  \hat Q_{k-1} ,m}(X_k)
     \end{equation}
    diverges to $\infty$ exponentially fast almost surely under the alternative $X_1,X_2 \dots \iid Q$:
    \begin{equation}\label{eqn:exp-growth-Q}
       \liminf_{n \to \infty}  \frac{\log  M_n^{(m)} }{n}  \ge  \frac{(f_Q(m+1) - f_Q(m))^2}{2( f_Q(m)+ f_Q(m+1)  )  } > 0;
    \end{equation}
    and is a supermartingale under the null $X_1,X_2 \dots \iid P \in \cM$.
\end{lemma}

The lemma above is based simply on the i.i.d.\ and martingale strong laws of large numbers, which we formally spell out in \cref{sec:pf-orcl}. Supermartingales with initial value 1 are e-processes simply due to the optional stopping theorem.
As is mentioned in \cref{sec:intro}, a power-one sequential test is formed by rejecting the null $\cM$ at the stopping time $\tau = \min\{ n : M_n^{(m)} \ge 1/\alpha \}$. $\tau$ is finite with probability 1 under $Q$ due to the exponential growth \eqref{eqn:exp-growth-Q}.
The most powerful e-process given a fixed alternative ground truth $Q$ is again discussed later in \cref{sec:ripr}, after developing the rich one-observation theory.

A \emph{mixture} argument over all possible $m \in \mathbb Z^{\ge 0}$ on the e-processes $(M_n^{(m)})$ constructed above then leads to the power-one sequential test for $\cM$ against $\cP(\mathbb Z^{\ge 0}) \setminus \cM$ we want, allowing exponential growth
when
we do not know the monotonicity-breaking point $m$ by oracle.

\begin{theorem}[Power-one sequential test of monotonicity]\label{thm:power1-monotone}
     Let $Q \in \cP(\mathbb Z^{\ge 0}) \setminus \cM$. Then, the process
     \begin{equation}\label{eqn:mixture-Mn}
      M_n = \sum_{m=0}^\infty 2^{-m-1} M_n^{(m)} 
     \end{equation}
     diverges to $\infty$ exponentially fast almost surely under the alternative $X_1,X_2 \dots \iid Q$:
    \begin{equation}\label{eqn:div-rate}
       \liminf_{n \to \infty}  \frac{\log  M_n }{n}  \ge \sup_{m \in \mathbb Z^{\ge 0} } \ \frac{((f_Q(m+1) - f_Q(m))^+)^2}{2( f_Q(m)+ f_Q(m+1)  )  } > 0;
    \end{equation}
    and is a supermartingale under the null $X_1,X_2 \dots \iid P \in \cM$.
\end{theorem}
The proof is straightforward by noting that $\log M_n \ge \log M_n^{(m)} - \log(2^{-m-1})$ for any $m$ where $f_Q(m+1) > f_Q(m)$. The supremum in \eqref{eqn:div-rate} is remarkable in that the test adapts to the unknown ``most antimonotone location'' of the underlying distribution. The rejection time $\tau = \min\{ n : M_n \ge 1/\alpha \}$ is now finite with probability 1 under any unknown $Q \in \cP(\mathbb Z^{\ge 0}) \setminus \cM$. It is also clear that the infinite sum \eqref{eqn:mixture-Mn} is always computable within $\mathcal{O}(n \max\{ X_1,\dots, X_n \})$ steps, as $M_n^{(m)} = 1$ for large $m$.

\subsection{Power-one e-process for $\cD_{\theta}$ against $\cD_{\theta}^c$}\label{sec:power1Dtheta}

We can use the same recipe of predictable streams of $E^{Q,m}$ to construct a power-one sequential test against the $\theta$-unimodality null. Here, let us extend these functions $E^{Q,m}$, which were previously defined on $\mathbb Z^{\ge 0}$ to the entirely of $\mathbb Z$, with the same definition as in \eqref{eqn:def-EQm}:
\begin{equation}
    E^{Q,m}(x) = 1 + \lambda_{Q,m}( - \id_{\{ x=m \}} + \id_{\{ x=m + 1 \}}).
\end{equation}
That is, when $m \ge 0$, $ E^{Q,m}(x) = 1$ for all $x < 0$. Testing $\cD_{\theta}$ now involves applying these wavelet functions to both $X_n - \theta$ and $\theta - X_n$. We directly state the mixture power-one e-process below omitting its straightforwad proof.

\begin{theorem}[Power-one sequential test of $\theta$-unimodality]\label{thm:power1-unimodal}
     Let $Q \in \cP(\mathbb Z) \setminus \cD_{\theta}$. Then, the process
     \begin{equation}\label{eqn:Jn}
      J_n(\theta) = \sum_{m=0}^\infty 2^{-m-2} (J_n^{m+}(\theta) +  J_n^{m-}(\theta))
     \end{equation}
     where
     \begin{equation}
         J_n^{m+}(\theta) = \prod_{k=1}^n E^{  \hat Q_{k-1} ,m}(X_k - \theta), \quad  J_n^{m-}(\theta) = \prod_{k=1}^n E^{  \hat Q_{k-1} ,m}(\theta - X_k)
     \end{equation}
     diverges to $\infty$ exponentially fast almost surely under the alternative $X_1,X_2 \dots \iid Q$:
    \begin{multline}\label{eqn:uni-growth}
       \liminf_{n \to \infty}  \frac{\log  J_n(\theta) }{n}  \ge \\ \max\left\{ \sup_{m \ge \theta} \ \frac{((f_Q(m+1) - f_Q(m))^+)^2}{2( f_Q(m)+ f_Q(m+1)  )  } , \sup_{m \le \theta} \ \frac{((f_Q(m-1) - f_Q(m))^+)^2}{2( f_Q(m)+ f_Q(m-1)  )  } \right\} > 0;
    \end{multline}
    and is a supermartingale under the null $X_1,X_2 \dots \iid P \in \cD_\theta$.
\end{theorem}

One naturally questions if unimodality (with unknown $\theta$), not just $\theta$-unimodality for a given $\theta$, can be similarly tested with power one on $\mathbb Z$. If a finite range for the modes is known \emph{a priori}, then the answer is positive by a simple finite infimum argument.

\begin{corollary}[Power-one sequential test of $\Theta$-unimodality] \label[corollary]{cor:Theta}
    Let $\Theta$ be a finite nonempty subset of $\mathbb Z$ and define $\cD_\Theta = \bigcup_{\theta \in \Theta} \cD_\theta$. Then, the process
     \begin{equation}\label{eqn:JnTheta}
      J_n(\Theta) = \min_{\theta \in \Theta} J_n (\theta)
     \end{equation}
     diverges to $\infty$ exponentially fast almost surely under any alternative $Q \in \cP(\mathbb Z) \setminus \cD_\Theta$ and is an e-process under the null $X_1,X_2,\dots \iid P \in \cD_\Theta$.
\end{corollary}

\cref{cor:Theta}, however, fails when $\Theta$ is infinite. To see this, $\lim_{\theta \to \pm \infty} J_n(\theta) = 1$ given arbitrary observations $X_1,\dots, X_n$. Therefore, $\inf_{\theta \in \Theta} J_n(\Theta) = 1$ always holds with infinite $\Theta$, precluding the analog of \cref{cor:Theta} that tests unimodality with completely unrestricted mode ($\Theta = \mathbb Z$). We shall later develop a power-one test for $\cD_{\mathbb Z} = \bigcup_{\theta\in\mathbb Z} \cD_\theta$ in \cref{sec:ooci} via the following two-step procedure: first, use the first observation $X_1$ to construct a strong (defined in the next subsection) finite $(1-2\alpha/3)$-confidence set $\CI(X_1)$ for mode; second, launch the e-process $J_n(\CI(X_1))$ from \cref{cor:Theta} with upcoming observations $X_2,X_3,\dots$ and reject at level $\alpha/3$. We defer this method until \cref{sec:ooci} when we introduce the complete one-observation theory of the problem including $\CI(X_1)$. Our results here, however, lead us to important multiple-observation confidence sets that we discuss in the next subsection.

\subsection{Confidence sequence and consistent mode estimator}

Confidence sets that are valid uniformly over time, called confidence \emph{sequences}, are routinely derived by inverting e-processes \citep[Chapter 13]{ramdas2024hypothesis}. While this also works with our e-processes $J_n(\theta)$ defined in \cref{thm:power1-unimodal}, we raise a caveat here on the important distinction between a \emph{strong} and \emph{weak} confidence set for mode, as modes of a unimodal distribution (per our definition $\cD_\theta$) need not be unique.

\begin{definition}[Weak vs.\ strong mode confidence sets] \label[definition]{def:all-modes-CI}
We denote by $\cD_{\mathbb Z} := \bigcup_{\theta \in \mathbb Z} \cD_\theta$ the set of all unimodal distributions on $\mathbb Z$. For $P \in \cD_{\mathbb Z}$, define $M(P)$ the set of all modes of $P$:
\begin{equation}
    M(P) = \{ \theta \in \mathbb Z : P \in \cD_\theta \}.
\end{equation}
Then,
$\CI_n$ is said to be a \underline{weak} $(1-\alpha)$-confidence set for mode if
\begin{equation}
  \forall P \in \cD_{\mathbb Z}, \forall \theta \in M(P),  P(\theta \in \CI_n) \ge 1-\alpha;
\end{equation}
$(\CI_n)$ is said to form a \underline{weak} $(1-\alpha)$-confidence sequence for mode if
\begin{equation}
  \forall P \in \cD_{\mathbb Z}, \forall \theta \in M(P),  P(\forall n, \theta \in \CI_n) \ge 1-\alpha.
\end{equation}
Further, $\CI_n$ is said to be a \underline{strong} $(1-\alpha)$-confidence set for mode if
\begin{equation}
  \forall P \in \cD_{\mathbb Z}, P(M(P) \subseteq \CI_n) \ge 1-\alpha;
\end{equation}
$(\CI_n)$ is said to form a \underline{strong} $(1-\alpha)$-confidence sequence for mode if
\begin{equation}
  \forall P \in \cD_{\mathbb Z}, P(\forall n, M(P) \subseteq \CI_n) \ge 1-\alpha.
\end{equation}
\end{definition}

It is easy to see that $M(P)$ is always a finite interval (i.e.\ a contiguous set) in $\mathbb Z$. Consequently, the convex hull (denoted by $\overline{A}$, e.g.\ $\overline{\{1,3\}} = \{1,2,3\}$) of a weak $(1-\alpha/2)$-confidence set (resp.\ sequence) is a strong $(1-\alpha)$-confidence set (resp.\ sequence) for mode, via a union bound on the smallest and the largest modes. If it is known \emph{a priori} that the underlying distribution contains exactly one mode, i.e.\ $|M(P)| = 1$, then it makes sense to use the definite form ``\emph{the} mode'' and weak confidence sets and sequences for mode become confidence sets and sequences for \emph{the} mode.

The following confidence sequence now follows from the standard construction.

\begin{corollary}
    The sets $\CI_n^\alpha = \{\theta \in \mathbb Z : J_n(\theta) \le 1/\alpha \}$ where $J_n(\theta)$ is defined in \cref{thm:power1-unimodal} form a weak $(1-\alpha)$-confidence sequence for mode. Consequently, the sets $\overline{\CI_n^{\alpha/2}}$ form a strong $(1-\alpha)$-confidence sequence for mode.
\end{corollary}

\begin{proof}
   This is due to Ville's inequality for e-processes \citep[Fact 7.6]{ramdas2024hypothesis}.
\end{proof}

A more interesting application of this routine is the following
 consistent \emph{set-valued mode estimator} that almost surely settles to the ground truth within bounded domains.
\begin{corollary}
    Define
    \begin{equation}
       \widehat M_n(P) = \{ \theta \in \mathbb Z : J_n(\theta) \le n^2 \}
    \end{equation}
    where $J_n(\theta)$ is defined in \eqref{eqn:Jn}. Let $P$ be a unimodal distribution on $\mathbb Z$ and $M(P)$ be its nonempty mode set.
    Then, almost surely, for every bounded subset $ B \subseteq \mathbb Z$:
    \begin{equation}
      \text{there exists } N(B) \text{ such that for all }n \ge N(B), \;     \widehat M_n(P) \cap B  = M(P)  \cap B.
    \end{equation}

\end{corollary} 
\begin{proof}
    First, let $\theta \in M(P)$. By Markov's inequality, $\Pr(\theta \in  \widehat M_n(P) ) \ge 1- 1/n^2$. Therefore by Borel-Cantelli lemma, $\theta \in  \widehat M_n(P)$ for all but finitely many $n$ almost surely. Since $M(P)$ is finite, we see that  $M(P) \subseteq  \widehat M_n(P) $  for all but finitely many $n$ almost surely. 
    
    Second, let $\theta \notin M(P)$. So $P \notin \cD_{\theta}$. Since $J_n(\theta)$ grows exponentially fast, $\theta \notin  \widehat m_n(P)$ for all but finitely many $n$ almost surely.  Since $B \setminus M(P)$ is finite,  we see that  $B \setminus M(P) \subseteq B \setminus  \widehat M_n(P)$  for all but finitely many $n$ almost surely.

    We have thus showed that  $B \setminus M(P) = B \setminus  \widehat M_n(P)$  for all but finitely many $n$ almost surely. The desired result follows from the fact that there are only countably many such subsets $B \subseteq \mathbb Z$.
\end{proof}

We remark that the bounded ``clipping set'' $B$ exists in the statement because our growth statement \eqref{eqn:uni-growth} does not rule out the possibility that different non-modes $\theta \notin M(P)$ escape from $ \widehat M_n(P) $ \emph{at unboundedly different times}. We do, however, speculate that the unclipped settlement: 
\begin{equation}
    \text{there exists } N \text{ such that for all }n \ge N, \;     \widehat M_n(P) =  M(P) 
\end{equation}
also holds true almost surely. To show this, we need a \emph{uniform growth} condition on the family of \emph{all alternative} test supermartingales $\{ (J_n(\theta)) : \theta \notin M(P) \}$, instead of each single one (which is the statement \eqref{eqn:uni-growth}). This is a rather profound problem and (together with the more general problem of ``how to control the size of confidence sets inverted from a family of e-values'') we do not pursue it in this paper.



\section{One-observation landscape of monotonicity and unimodality on $\mathbb Z$}\label{sec:1obs}

We have so far developed some testing and estimation methods for $\cM$ and $\cD_{\theta}$ given an infinite stream of observations. We have also left open some questions that either may strengthen these sequential results or are of independent interest. What are the e-value sets $\fM$ and $\fD_{\theta}$? Are there more e-powerful e-values than the ``wavelets'' in \cref{lem:log-power-monotone-evalue}? How to construct the finite confidence set $\CI(X_1)$ mentioned in the remark following \cref{cor:Theta}?
These questions 
warrant a close-up study of the \emph{one}-observation landscape of distribution classes $\cM$ and $\cD_{\theta}$.

Our exposition in this section draws heavily upon the convex analysis and duality machinery developed recently by \cite{numeraire}. Let us introduce some additional notations.

\paragraph{Notations, continued.} 
We begin to consider measures that are not necessarily probabilities. The set of all $\sigma$-finite measures over $\mathbb A$ is denoted by $\cS(\mathbb A)$, and the set of all subprobability measures by $\cP^{-} (\mathbb A)$. For $P \in \cS(\mathbb Z)$, we still denote by $f_{P}$ the measure mass function (MMF) of $P$, i.e.\ $f_{P}(n) = P(\{n\})$. Subsets of $\cS(\mathbb A)$ are denoted by caligraphic letters $\cA, \cB,$ etc., and subsets of $\fF(\mathbb A)$ by boldface letters $\fA, \fB$, etc. We also adopt the following operators that map back and forth between $\cS(\mathbb A)$ and $\fF(\mathbb A)$.
\begin{definition}[Polar, prepolar, and bipolar]
    Let $\cA \subseteq \cS(\mathbb A)$. Define
    \begin{equation}
        \cA^\circ = \left\{ f \ge 0 : \sup_{P \in \cA} \int f(x)P(\d x) \le 1  \right\}.
    \end{equation}
    Let $\fA \subseteq \fF(\mathbb A)$. Define
    \begin{equation}
        {^\circ \fA} = \left\{ P \in  \cS(\mathbb A): \sup_{f \in \fA } \int f(x)P(\d x) \le 1  \right\}, \quad ^\circ_1 \fA =  ({^\circ\fA}) \cap \cP(\mathbb A).
    \end{equation}
\end{definition} 
The sets $\cA^\circ$ and ${^\circ \fA}$ are called \emph{polar} and \emph{prepolar} in the convex analysis literature. See \cite{numeraire} for example. In particular, if $\cA \subseteq \cP(\mathbb A)$, $\cA^\circ$ is the set of \emph{all} one-observation e-values for $\cA$. Dually, $^\circ_1 \fA$ is the set of all probability distributions for which all functions in $\fA$ are e-values. The set ${^\circ(\cA^\circ)}$ obtained from concatenating both operations is called the \emph{bipolar} of the set $\cA$, and ${^\circ_1(\cA^\circ)}$, the set of all probabilities in the bipolar, is of statistical interest as it contains distributions indistinguishable from $\cA$: an e-value for $\cA$ is always an e-value for ${^\circ_1(\cA^\circ)}$ under any $Q \in {^\circ_1(\cA^\circ)} \setminus \cA$. In fact, any \emph{test} with type-I error level $\alpha$ for the null $\cA$ has type-II error at least $1-\alpha$ under any $Q \in {^\circ_1(\cA^\circ)} \setminus \cA$, because every test at type-I error level $\alpha$ corresponds to the level set of some e-value \citep[Fact 3.6]{ramdas2024hypothesis}. Finally, it is easy to see that (1) $\cA \subseteq {^\circ(\cA^\circ)}$ and $\fA \subseteq ({^\circ \fA})^\circ$, (2) $\cA \subseteq \cB \implies \cB^\circ \subseteq \cA^\circ$ and $\fA \subseteq \fB \implies {^\circ\fB} \subseteq {^\circ \fA}$.


\subsection{Polar of $\cM$}\label{sec:monotone}

We investigate $\fM = \cM^\circ $, the set of all one-observation e-values for (i.e.\ polar of) the monotone class
\begin{equation}
    \cM = \left\{ P \in \cP(\mathbb Z^{\ge 0}) : f_P(n) \ge   f_P (n+1) \text{ for all }n \in  \mathbb Z^{\ge 0}   \right\}.
\end{equation}
To this end, let us define the set of all nonnegative, increasing, sub-diagonal functions on $\mathbb Z^{\ge 0}$:
\begin{equation}
    \fR:= \{\rho \in \fF(\mathbb Z^{\ge 0}) : 0 = \rho(0) \le \rho(1) \le \dots; \rho(n) \leq n \text{ for all } n \}
\end{equation}
The function class $ \fM \subseteq \fF(\mathbb Z^{\ge 0})$ that is the polar of $\cM$ turns out to be the \emph{finite differences} (i.e.\ discrete derivatives) of those functions in $\mathbf R$:
\begin{equation}\label{eqn:set-M}
    \left\{ n \mapsto \rho(n+1) -\rho(n) : \ \rho \in \fR \right\}.
\end{equation}
In words, every $\rho \in \mathbf R$ gives rise to a corresponding e-value $\Delta\rho \in \fM $ by taking differences. The equality of $\fM$ and $\cM^\circ$ is proved below. To reduce parentheses clutter we write $\rho_n$ for $\rho(n)$ when $\rho \in \fR$.

\begin{theorem}\label{thm:polar-of-M}
    $\fM =  \left\{ n \mapsto \rho(n+1) -\rho(n) : \ \rho \in \fR \right\}$. That is, functions in the form of $\rho_{n+1} -\rho_n$ are exactly all e-values testing monotonicity.
\end{theorem}
\begin{proof} Let $\Delta \fR = \left\{ n \mapsto \rho(n+1) -\rho(n) : \ \rho \in \fR \right\}$.
First, suppose $E \in \Delta \fR$ with $E = \Delta \rho$, $\rho \in \fR$. For any $P \in \cM$, we have
\begin{align}
    & \sum_{n=0}^{N-1} (\rho_{n+1} - \rho_{n}) f_P(n) =   \sum_{n=0}^{N-1} \{ (\rho_{n+1} - (n+1)) - (\rho_n - n) +  1\} f_P(n)
    \\
    = & \sum_{n=0}^{N-1} f_P(n) + \sum_{n=0}^{N-1} \{ (\rho_{n+1} - (n+1))f_P(n) - (\rho_n - n)f_P(n)  \}
    \\
    \le &  \sum_{n=0}^{N-1} f_P(n) + \sum_{n=0}^{N-1}\{ (\rho_{n+1} - (n+1))f_P(n+1) - (\rho_n - n)f_P(n)  \}
    \\
    = &  \sum_{n=0}^{N-1} f_P(n) +  (\rho_N - N) f_P(N) - \rho_0 f_P(0) \le 1.
\end{align}
So $\Expw_{X \sim P}(E(X)) = \sum_{n=0}^{\infty} (\rho_{n+1} - \rho_{n}) f_P(n) \le 1$. Therefore, $E \in \cM^\circ$. This shows that $\Delta \fR \subseteq \cM^\circ$.

Second, suppose $E \in \cM^\circ$. Define the increasing sequence
    \begin{equation}
        \rho_{n} = \sum_{k=0}^{n-1} E(k)
    \end{equation}
    which satisfies $E = \Delta \rho$ and $\rho_0 = 0$.
    Now, for any $n \in \mathbb Z^{\ge 0}$, consider $P_n = \operatorname{unif}_{\{ 0,\dots, n-1 \}} \in \cM$. Since $E \in \cM^\circ$, $E$ is an e-value for $P_n$:
    \begin{equation}
        1 \ge \sum_{k=0}^\infty E(k) f_{P_n}(k) = \frac{1}{n}  \sum_{k=0}^{n-1} E(k) =  \frac{\rho_n}{n},
    \end{equation}
    concluding that $ \rho_{n}  \le n$. Therefore $\rho \in \fR$, $E \in \Delta \fR$, and consequently $\cM^\circ \subseteq \Delta \fR$.
\end{proof}

We elaborate in \cref{sec:constraints} on how we constructively discover the function set $\fM$ which contains exactly all e-values for $\cM$ behind the scenes. In \cref{sec:power,sec:x+1qx}, we discuss some properties the polar set $\fM = \cM^\circ$.

\subsection{Power, bipolar, and testability of $\cM$}\label{sec:power}

Every e-value $E(X)$ for a distribution $\cA$ naturally leads to a level-$\alpha$ test where $\cH_0: P \in \cA$ is rejected if $E(X) \ge 1/\alpha$ upon observing $X \sim P$. This guarantees a type-I error at most $\alpha$ due to Markov's inequality. Conversely, if $T(X) \in \{ 0, 1\}$ is a test with type-I error at most $\alpha$:
\begin{equation}
    \sup_{P \in \cA} P( \{ x:T(x) = 1\} ) \le \alpha,
\end{equation}
the ``all-or-nothing'' e-value $\alpha^{-1} T(X)$ leads to the exact same test.
Thus, the characterization of the set $\fM = \cM^\circ$ describes all tests for monotonicity: $T:\mathbb Z^{\ge 0} \to \{ 0, 1\}$ is a valid test of level-$\alpha$ for the null $\cM$ if and only if $T(x) = \id_{ \{ E(x) \ge 1/\alpha  \} }$ for some $E \in \fM$.

However, this
does not answer if any of these tests are powerful if the data-generating distribution $P \notin \cM$. To this end, we need to consider the prepolar of $\fM$ i.e.\ the bipolar of $\cM$.
If there exists a probability distribution $P \in {^\circ_1\fM} \setminus \cM$, then on a fundamental level, $P$, though non-monotone, is indistinguishable from monotone distributions; no level-$\alpha$ test testing the null of monotonicity may reject the null with probability $>\alpha$ if the true distribution is $P$. The following theorem answers this question in the negative. 

\begin{theorem}[Strong testability of monotonicity]\label{thm:bipolar-monotone}
    $^\circ\fM = {^\circ(\cM^\circ)}$, the bipolar of $\cM$, is the set of all subprobabilities on $\mathbb Z^{\ge 0}$ that are bounded by some monotone probability:
    \begin{equation}
        \cM^- = \{ P \in \cP^{-}(\mathbb Z^{\ge 0}) : P \le P' \text{ for some }P' \in  \cM  \}.
    \end{equation}
    Therefore, ${^\circ_1(\cM^\circ)} = \cM$, meaning that there is no non-monotone distribution that is indistinguishable from monotone distributions on $\mathbb Z^{\ge 0}$.
\end{theorem}

In a word, monotonicity is a property such that, as a set of probabilities, subprobabilities indistinguishable from it by one observation are necessarily dominated by some member probability in it. We call monotonicity ``strongly testable'' in this sense.
We prove \cref{thm:bipolar-monotone} in \cref{sec:pf-bipolar-M}. The key idea is that, if $P \in {^\circ\fM}$, we consider the ``backward running supremum'' of the PMF of $P$:
\begin{equation}
    f^*(n) = \sup_{k \ge n} f_P(n),
\end{equation}
and show that it is a PMF of some subprobability. On the other hand,  the weaker statement ${^\circ_1(\cM^\circ)} = \cM$ is much easier to prove: if $Q \in \cP(\mathbb Z^{\ge 0}) \setminus \cM$, it suffices to find a function $E \in \fM$ that is not an e-value under $Q$ (``witnesses'' $Q$). Suppose $f_Q(m) < f_Q(m+1)$. One such witness function is the $E^{Q,m}$  defined back in \cref{lem:log-power-monotone-evalue}.

\subsection{The $(X+1)q(X)$ e-values for $\cM$}\label{sec:x+1qx}

We note an interesting subset of $\fM$: functions of the form $(n+1) q(n)$ where $q$ is any (sub)probability mass function. Formally, define the following subset of $\fF( \mathbb Z^{\ge 0} )$,
\begin{equation}
    \fQ = \{ n \mapsto (n+1) f_Q(n) : Q \in \cP^-(\mathbb Z^{\ge 0})  \},
\end{equation}
where the PMF of a subprobability $Q$ is also denoted as $f_Q$.
That is, $\fQ$ consists of all functions that are a product of $n+1$ and some sub-PMF on nonnegative integers.

\begin{proposition}\label[proposition]{thm:discr}
   $\fQ \subseteq \fM$. That is, functions in $\fQ$ are e-values testing monotonicity.
\end{proposition}
\begin{proof}
 Let $E(n) = (n+1) f_Q (n) \in \fQ$ with $Q \in \cP^-(\mathbb Z^{\ge 0})$. With
\begin{equation}
    \rho_n =  \sum_{k=0}^{n-1}(k+1)f_Q(k) =\sum_{k=0}^{n-1} \sum_{\ell = k}^{n-1} f_Q(
\ell
    )   \le \sum_{k=0}^{n-1} 1   = n
\end{equation}
we see that $\rho \in \fR$ and $E = \Delta \rho \in \fM$. 
\end{proof}

The e-value class $\fQ$ is of interest not only because of the succinct expression $(X+1)q(X)$, but also due to its connection to a ``basic inequality'' for monotone distributions:
for any $P \in \cM$, the monotonicity of $f_P$ implies $f_P(n) \le \frac{f_P(0) + \dots + f_P(n)}{n+1} \le \frac{1}{n+1}$ for all $n$. Formally let us define
    \begin{equation}\label{eqn:discrete-bi}
       \cM' = \left\{  P \in \cS(\mathbb Z^{\ge 0}) :  f_P(n)  \le \frac{1}{n+1} \text{ for all }n \in  \mathbb Z^{\ge 0} \right\} \supseteq \cM.
    \end{equation}
The reason for calling it the ``basic inequality'' shall transpire when we later present the continuous analog of this lemma, \cref{lem:basic}, which forms the foundation for many existing results on unimodal distributions in the literature. It is also worth noting that $\cM$ is a proper subset of $\cM' \cap \cP(\mathbb Z^{\ge 0})$. For example, the distribution on $\{ 0,1,2 \}$ with the non-monotone PMF $p(0) = 1/2, p(1) = 3/14, p(2)=4/14$ belongs to $\cM' \cap \cP(\mathbb Z^{\ge 0})$ but not to $\cM$.

The following statement reveals the bipolar relation between measures satisfying the basic inequality $f_P(n) \le 1/(1+n)$ and functions taking $(x+1)q(x)$ form. We prove it in \cref{sec:pf-xqx-bipolar}.

\begin{proposition}\label[proposition]{prop:xqx-bipolar}
    $\cM' = {^\circ\fQ}$ and $\cM'^\circ = \fQ$.
\end{proposition}

\subsection{Polar of $\cD_{\theta}$}\label{sec:unimodal}

Monotonicity, whose one-observation landscape we have completely understood from \cref{sec:monotone,sec:power} above, can be seen as half of the conditions for unimodality, which we study in this section. Recall that $\fD_\theta$ is the polar of the $\theta$-unimodal class
\begin{equation}
    \cD_{\theta} = \left\{ P \in \cP(\mathbb Z) :  \text{$f_P(n-1) \le f_P(n)$ if $n \le \theta$ and $f_P(n) \le f_P(n+1)$ if $n \ge \theta$}  \right\}.
\end{equation}
It is worth noting that $\cD_{\theta}$ cannot be formed via elementary transformations from the previously studied monotone class $\cM$. For example, it is tempting to consider the class
\begin{equation}
    \{ L(X) : L((X-\theta)^+) \in \cM, L((\theta - X)^+) \in \cM \}
\end{equation}
($L(X)$ denotes the distribution of a random variable $X$)
which turns out to be a strict superset of $\cD_{\theta}$ as they differ at the mode $\theta$. Therefore, results on the polar and bipolar of $\cM$ proved in \cref{sec:monotone,sec:power} cannot be imported directly to yield those on $\cD_{\theta}$. 

Still, we can use methods similar to the ones used in \cref{sec:monotone,sec:power} to derive analogous statements on the polar and bipolar of the unimodal class $\cD_{\theta}$. We first present $\cD_{\theta}^\circ$, the set of all e-values for $\theta$-unimodal distributions over $\mathbb Z$. 

\begin{theorem}\label{thm:polar-unimodal-Z}
    The polar set $\fD_\theta = \cD_{\theta} ^\circ$, i.e.\ all e-values testing $\theta$-unimodality can be represented as follows:
    \begin{equation}\label{eqn:set-Dtheta}
        \fD_{\theta} = \left\{ n \mapsto \begin{cases}
            \rho_{n-\theta + 1} -\rho_{n-\theta}, & n > \theta \\
            \rho_1 + \eta_1 - 1, &  n = \theta \\
            \eta_{\theta - n + 1} - \eta_{\theta - n}, & n < \theta
        \end{cases} \;  : \;  
        \begin{matrix}
\rho_1 + \eta_1 \ge 1; \ \rho_1 \le \dots \le  \rho_n \le n \\ \text{and }
 \eta_1 \le \dots \le  \eta_n \le n
       \text{ for all } n 
\end{matrix}
 \right\}.
    \end{equation}
\end{theorem}

It is worth noting that although not stated explicitly, $\rho_1 \ge 0$ and $\eta_1 \ge 0$ must be satisfied, implied by the description in \eqref{eqn:set-Dtheta}.
We prove \cref{thm:polar-unimodal-Z} in \cref{sec:pf-polar-Z}. The way we construct this polar set is also revealed in \cref{sec:constraints}. 

\subsection{Power, bipolar, and testability of $\cD_{\theta}$}\label{sec:power-unimodal}

Like the monotone class, the unimodal class $\cD_{\theta}$ also enjoys the following bipolar statement that establishes its strong testability.
\begin{theorem}[Strong testability of $\theta$-unimodality]\label{thm:bipolar-unimodal}
    $^\circ\fD_{\theta} = {^\circ(\cD_{\theta}^\circ)}$, the bipolar of $\cD_{\theta}$, is the set of all subprobabilities on $\mathbb Z$ that are bounded by some $\theta$-unimodal probability:
    \begin{equation}
        \cD_{\theta}^- = \{ P \in \cP^{-}(\mathbb Z) : P \le P' \text{ for some }P' \in  \cD_{\theta} \}.
    \end{equation}
    Therefore, ${^\circ_1( \cD_{\theta}^\circ)} =  \cD_{\theta}$, meaning that there is no non-$\theta$-unimodal distribution that is indistinguishable from $\theta$-unimodal distributions on $\mathbb Z$.
\end{theorem}
The proof of \cref{thm:bipolar-unimodal} is a slightly more involved generalization of the proof of \cref{thm:bipolar-monotone}. We spell out the proof in \cref{sec:pf-bipolar-D}. 
Similarly, the weaker statement ${^\circ_1( \cD_{\theta}^\circ)} =  \cD_{\theta}$ can be seen straightforwardly from the following:  if $Q \in \cP(\mathbb Z) \setminus \cD_{\theta}$, a function $E \in \fD_\theta$ that witnesses $Q$ is
\begin{equation}
    E(n) = 1 - \id_{\{ n=m \}} + \id_{\{ n=m + 1 \}}
\end{equation}
if $f_Q(m) < f_Q(m+1)$, $m \ge \theta$; or
\begin{equation}
    E(n) = 1 - \id_{\{ n=m \}} + \id_{\{ n=m - 1 \}}
\end{equation}
if $f_Q(m) < f_Q(m - 1)$, $m \le \theta$.

\subsection{Acceptance and confidence sets, and power-one $\mathbb Z$-unimodality test}\label{sec:ooci}

A very intriguing consequence of \cref{thm:polar-unimodal-Z} is the following inevitability in testing unimodality: \emph{$\theta$-unimodality cannot be rejected based on the largeness of a single observation}. Formally:
\begin{corollary}
    Let $\alpha \in (0,1)$ and $T: \mathbb Z \to \{ 0, 1 \}$ be a level-$\alpha$ valid one-observation test of $\theta$-unimodality. That is,
    \begin{equation}
       P( T^{-1}(1) ) \le \alpha, \quad \forall P \in  \cD_{\theta}.
    \end{equation}
    Then, $\sup P( T^{-1}(0) ) = \infty$ and $\inf  P( T^{-1}(0) ) = -\infty$, meaning that $T$ must accept arbitrarily large positive and negative observations. 
\end{corollary}
\begin{proof}
    Consider the all-or-nothing e-value associated with $T$: $E(n) = \alpha^{-1}T(n)$, which belongs to $\fD_\theta$, thus admitting a representation as in \eqref{eqn:set-Dtheta} with sequences $\{\rho_n\}$ and $\{\eta_n \}$.
    
     Assume on the contrary, $ \sup P( T^{-1}(0) ) = \theta + M -1$. Then, for all $n \ge \theta + M$, $E(n) = \alpha^{-1}$. Let $K > M/(\alpha^{-1} - 1)$ be a large integer. We have,
    \begin{equation}
       \rho_{M + K} = \rho_M + \sum_{n  =\theta + M}^{\theta + M + K - 1} E(n) \ge  \sum_{n  =\theta + M}^{\theta + M + K - 1} E(n) \ge K \alpha^{-1} > M + K,
    \end{equation}
    contradicting the condition $ \rho_{M + K} \le M + K$. Therefore, $\sup P( T^{-1}(0) ) = \infty$. The other direction  $\inf P( T^{-1}(0) ) = -\infty$ is analogous.
\end{proof}

While any one-observation test of $\theta$-unimodality must therefore have a (bidirectionally) \emph{unbounded acceptance set}, it is possible, we note, to construct a \emph{bounded confidence set} for the mode based on one unimodal observation. This fact resonates with the results of \cite{Edelman01111990,Andel} in the continuous case.

In fact, finite confidence sets for mode (weak and strong alike, recall \cref{def:all-modes-CI}) can be constructed via the $(X+1)q(X)$-shaped e-values discussed in \cref{sec:x+1qx}. We first note that
\cref{thm:discr} implies the following results on testing $\theta$-unimodality with e-values and constructing confidence intervals for the mode $\theta$ with one observation.
\begin{proposition}\label[proposition]{cor:test-Z}
    Let $\{q_\theta : \theta \in \mathbb Z \}$ be a family of subprobability mass functions on $\mathbb Z$ such that each $q_\theta$ is supported either on $\mathbb Z^{\ge \theta}$ or $\mathbb Z^{\le \theta}$.
    Let \begin{gather}
        E(X;\theta) = ( |X-\theta| + 1) \cdot q_\theta(X),
    \end{gather}
    Then each $ E(X;\theta)$ is an e-value for $\cD_{\theta}$. 
    Consequently,
    \begin{equation}\label{eqn:1obs-ci}
     \CI =   \{ \theta \in \mathbb Z :    E(X;\theta) < 1/\alpha  \}
    \end{equation}
    is a weak $(1-\alpha)$-confidence interval for mode, and thus a strong $(1-2\alpha)$-confidence interval for mode.
\end{proposition}

Then, we have the following instance of \cref{cor:test-Z} where $\CI$ is finite for all but one input value. An always finite CI can subsequently be obtained from a simple union bound.

\begin{corollary}\label[corollary]{cor:1obs-CI}
     Define $T_\alpha = 2/\alpha + 1$. For any $\phi \in \mathbb Z$,  the set
     \begin{equation}
        \CI(X;\alpha, \phi) = \begin{cases}
              ( X - T_\alpha |X - \phi|, X + T_\alpha |X -  \phi| ) \cap \mathbb Z , & X \neq \phi; \\
              \mathbb Z , & X = \phi.
        \end{cases}
    \end{equation}
     is a weak $(1-\alpha)$-confidence interval for mode. Consequently, if $\phi \neq 0$,
      \begin{equation}
         \CI(X;\alpha/2, \phi)  \cap \CI(X;\alpha/2, -\phi) 
    \end{equation}
    is a weak $(1-\alpha)$-confidence interval for mode that is always finite.
\end{corollary}

We prove \cref{cor:1obs-CI} in \cref{sec:pfci}. In our proof, we select $q_\theta$ to be the PMF of some discrete uniform distribution near (but excluding) 0.
The parameter $\phi \in \mathbb Z$ in this confidence set must be chosen in advance. One cannot first observe $X$ then pick a $\theta \neq X$ to avoid $\CI(X; \alpha, \phi)$ becoming $\mathbb Z$. Doing so violates the validity of the confidence set. Therefore the union bound set $\CI(X;\alpha/2, \phi)  \cap \CI(X;\alpha/2, -\phi)$ is recommended as it guarantees finiteness for any preselected $\phi$ and subsequently observed $X$.

It is also very worth remarking here that the one-observation confidence set derived by \cite{Edelman01111990} in the \emph{continuous} case reads  $(X - (T_\alpha -2) |X - \phi|, X + (T_\alpha -2) |X + \phi|)$. See \cref{prop:edelci}. Aside from a slightly smaller radius when $X \neq \phi$, this confidence set is a singleton $\{ X \}$ if the observation $X$ happens to be $\phi$. In his continuous case, this is no problem as $X = \phi$ only happens with probability 0. However, we can no longer allow this corner case unchecked as $X = \phi$ can now happen with positive probability in our discrete case. Therefore, the addition ``$\CI = \mathbb Z$ when $X = \phi$''  highlights a very fundamental distinction between discrete and continuous mode estimation problems.

Finally, we return to the unfinished task in \cref{sec:power1Dtheta}. We may use the finite one-observation mode confidence set \cref{cor:1obs-CI} to preset the ``mode range'' $\Theta$ in \cref{cor:Theta} to construct a power-one sequential test for $\mathbb Z$-unimodality (i.e.\ unimodality with unrestricted unknown mode).

\begin{theorem}[Power-one sequential test for unimodality]
    Define the following sequential test for $\cD_{\mathbb Z} = \bigcup_{\theta \in \mathbb Z} \cD_\theta$: reject the null hypothesis upon the first $n \ge 2$ such that
    \begin{equation}
        J_{2:n}( \CI(X_1;\alpha/3, \phi)  \cap \CI(X_1;\alpha/3, -\phi)   ) \ge 3/\alpha.
    \end{equation}
    Here, $\phi \neq 0$ is an integer chosen in advance, the one-observation confidence interval $\CI(\cdot)$ is the one defined in \cref{cor:1obs-CI},
    and
    $J_{2:n}(\Theta)$ denotes the e-process $J_n(\Theta)$ in \cref{cor:Theta} run on $(X_2,\dots, X_n )$. Then, this is a sequential test with uniform type-I error control below $\alpha$ and power 1. 
\end{theorem}

\begin{proof} Denote $\CI(X_1) =  \CI(X_1;\alpha/3, \phi)  \cap \CI(X_1;\alpha/3, -\phi)$.
    First, for any null distribution $P \in \cD_{\mathbb Z}$, by \cref{cor:1obs-CI} and the remark after \cref{def:all-modes-CI}, $P( M(P) \subseteq \CI(X_1)  ) \ge 1-2\alpha/3$. Conditioned on the event $\{ M(P) \subseteq    \CI(X_1)  \}$, $(J_{2:n}(\CI(X_1)))$ is an e-process, so $P(\text{rejection} |  M(P) \subseteq \CI(X_1) ) \le \alpha/3$. A union bound then gives $P(\text{rejection}) \le \alpha$. Second, for any alternative $Q \in \cP(\mathbb Z) \setminus  \cD_{\mathbb Z}$, $Q$ must also belong to $ \cP(\mathbb Z) \setminus  \cD_{\CI(X_1)}$. Then from \cref{cor:Theta} we know that $(J_{2:n}(\CI(X_1)))$ diverges to $\infty$ exponentially fast.
This concludes the proof.
\end{proof}

\subsection{Optimal e-value via least concave majorant}\label{sec:ripr}

The construction of the power-one sequential test supermartingale in \cref{sec:seq} is based on the building block of \cref{lem:log-power-monotone-evalue}. That is,
for an alternative probability measure $Q \in \cP(\mathbb Z^{\ge 0}) \setminus \cM$, we found a function $ E^{Q,m} \in \fM = \cM^\circ$ such that
\begin{equation}
   \Expw_{X \sim Q} \log E^{Q,m}(X) \ge \frac{(f_Q(m+1) - f_Q(m))^2}{2( f_Q(m)+ f_Q(m+1)  )  }  > 0.
\end{equation}
The expected log-value $ \Expw_{X \sim Q} \log E^{Q,m}(X)$ above, i.e.\ the \emph{e-power}, also decides the ``divergence rate'' in the downstream results \cref{lem:oracle-power1,thm:power1-monotone,thm:power1-unimodal}.

It is also worth noting, however, that $E^{Q,m}$ is not the the function $E \in \fM$ that maximizes the e-power $ \Expw_{X \sim Q} \log E(X)$; and  $ \frac{(f_Q(m+1) - f_Q(m))^2}{2( f_Q(m)+ f_Q(m+1)  )  }$ is not the maximal e-power  $\sup_{ E \in \fM } \Expw_{X \sim Q} \log E(X)$. To see this, in the proof of \cref{lem:log-power-monotone-evalue}, we employ the bound $\log(1+x) \ge x - x^2$ (which is loose for small $x$), and then pick $\lambda_{Q,m}$ that maximizes the resulting quadratic. Therefore, while the construction $E^{Q,m}$ in \cref{sec:seq} does guarantee power-one testing against $\cM$ and $\cD_{\theta}$ with exponentially growing test supermartingales under any alternative distribution, it is of interest to characterize
\begin{equation}\label{eqn:max-epower}
    \sup_{E \in \fM}  \Expw_{X \sim Q} \log E(X),
\end{equation}
which is the rate of growth of the supermartingale. \cite{numeraire} prove that, for a distribution class $\cA$ and an alternative $Q$ that satisfies $Q \ll \cA$ (a mild assumption that always holds in our case where $\cA = \cM$), the supremum $
      \sup_{E \in \cA^\circ}  \Expw_{X \sim Q} \log E(X)$ 
is actually always achieved by a $Q$-almost surely unique e-value called the \emph{numeraire} $E^*$ which satisfies the strong duality:
\begin{equation}\label{eq:strong-duality}
 \Expw_{X \sim Q} \log E^*(X) = \sup_{E \in \cA^\circ}  \Expw_{X \sim Q} \log E(X) =  \inf_{ P \in  {^\circ(\cA^\circ)}}  \kl( Q \| P ) = \kl( Q \| P^* ),
\end{equation}
where $P^*$ is a subprobability measure called the \emph{reverse information projection} (RIPr) of $Q$ onto $\cM$ (defined by $\d P^*/\d Q = 1/E^*$), even if all quantities are $+\infty$. In fact, the numeraire $E^*$ is uniquely characterized by the \emph{log-optimality} criterion
\begin{equation}\label{eqn:log-optimality-def}
    \Expw_{X \sim Q} \left\{ \log \frac{E(X)}{E^*(X)} \right\} \in [-\infty, 0] \quad \text{for all $E \in \cA^\circ$},
\end{equation}
which is an even stronger claim than the first equality in~\eqref{eq:strong-duality} (due to its ability to distinguish different infinities);
and similarly,
the RIPr $P^*$ is uniquely characterized by the \emph{relative KL minimization} criterion
\begin{equation}\label{eqn:log-comparison-rn}
    \Expw_{X \sim Q} \left\{ \log \frac{\d P^a}{\d P^*}(X) \right\} \in [-\infty, 0] \quad \text{for all $P \in {^\circ(\cA^\circ)}$},
\end{equation}
where $P^a$ denotes the absolutely continuous part of the measure $P$ with respect to $Q$, which is an even stronger claim than the final equality in~\eqref{eq:strong-duality} (same reason).  
Additionally, a (strikingly) simple criterion for \emph{finding}\footnote{more accurately, \emph{verifying} that a candidate $P^*$ is the RIPr} the RIPr $P^*$ and therefore the numeraire is reviewed in \cref{sec:tensorization} as \cref{prop:fundamental-numeraire}.

Now, moving back to the monotone testing problem, we
recall that the \emph{least concave majorant} (LCM) of a probability $Q$ is the probability $\tilde Q$ whose PMF $F_{\tilde Q}$ is the smallest concave function majorizing $F_Q$. It turns out that $\tilde Q$, after taking the $Q$-absolutely continuous part, is the RIPr of $Q$ onto the monotone class $\cM$.

\begin{theorem} \label{thm:ripr-lcm} Let $\tilde Q$ be the least concave majorant of $Q$. Then, $Q \ll \tilde Q$, and $\tilde Q^a$, the absolutely continuous part of $\tilde Q$ with respect to $Q$, i.e.\ the subprobability with mass function $f_{\tilde Q^a}(n) = f_{\tilde Q}(n) \id_{ \{ f_Q(n) > 0 \} }$, is the {reverse information projection} of $Q$ onto $\cM$. Consequently,
    the maximum \eqref{eqn:max-epower} is attained by the function
    \begin{equation}
        E^*_Q(n) = \frac{f_Q (n) }{f_{\tilde Q} (n)}
    \end{equation}
    where $0/0$ is understood as 0.
\end{theorem}
\begin{proof}
As the LCM, there exist $0 = N_0 < N_1 < \dots$ such that
    \begin{gather}
        F_Q(N_k) = F_{\tilde Q}(N_{k}), \quad F_Q(N_{k+1}) = F_{\tilde Q}(N_{k+1}), \\
        F_{\tilde Q} \text{ is linear and } F_{\tilde Q} \ge F_Q \text{ on } \{ N_k, N_k+1,\dots, N_{k+1} \}.
    \end{gather}
    The piecewise linearity of $F_{\tilde Q}$ implies that $f_{\tilde Q}(n)$ is always $\frac{F_Q(N_{k+1}) - F_Q(N_k)}{N_{k+1} - N_{k}}$ on $n \in \{  N_k + 1, \dots, N_{k+1} \}$.
    Therefore, if $f_{\tilde Q}(n) = 0$, so must be $f_Q(n)$, implying that $Q \ll \tilde Q$.

Note that the effect null while testing $\cM$ is the bipolar $\cM^-$ in \cref{thm:bipolar-monotone}, consisting of subprobabilities dominated by some monotone probability. To prove that $E_Q^* = f_Q/f_{\tilde Q} = f_Q / f_{\tilde Q^a}$ is the numeraire,
due to \citet[Theorem 3.6]{numeraire} (see \cref{prop:fundamental-numeraire}), it suffices to show that $\tilde Q^a$, the absolutely continuous part of $\tilde Q$, satisfies
 \begin{equation}
    \int \frac{\d {P^a}}{\d \tilde Q^a }(n) Q(\d n)  =    \sum_{f_Q (n) > 0 } f_Q(n) \frac{f_{P} (n) }{f_{\tilde Q}(n)} \le 1
    \end{equation}
    for any $P \in \cM^-$ whose absolutely continuous part w.r.t.\ $Q$ is $P^a$.
    This, notably, is equivalent to $f_Q/f_{\tilde Q} \in \fM$. In light of \eqref{eqn:set-M}, it suffices to show that $\sum_{n=0}^N f_Q (n) /f_{\tilde Q} (n) \le N+1$ for any $N$. In fact, we now show that $\sum_{n=0}^N f_Q (n) /f_{\tilde Q} (n) \le N+1$  holds \emph{even when in these fractions, $0/0$ is treated as $1$ instead of 0}.  Assuming $N_K < N \le N_{K+1}$ and using the piecewise linearity property of $\tilde Q$:
    \begin{align}
        & \sum_{n=1}^N \frac{f_Q (n)}{f_{\tilde Q} (n)} = \sum_{k=0}^{K-1} \sum_{n=N_k+1}^{N_{k+1}} \frac{f_Q (n)}{f_{\tilde Q} (n)} + \sum_{n=N_{K}+1}^N \frac{f_Q (n)}{f_{\tilde Q} (n)} \\
        =& \sum_{k=0}^{K-1}(N_{k+1} - N_k)   +  \frac{F_Q(N) - F_Q (N_{K})}{ f_Q (N_K+1) } \\ 
        \le& \sum_{k=0}^{K-1}(N_{k+1} - N_k)  +  \frac{F_{\tilde Q}(N) - F_Q (N_{K})}{ f_Q (N_K+1) }  = N.
    \end{align}
    Combining with $f_Q(0)/f_{\tilde Q}(0) = F_Q(0)/F_{\tilde Q} (0) = 1$, we obtain that $f_Q/f_{\tilde Q} \in \fM$, so $\tilde Q^a$ is the reverse information projection leading to log-optimal and thus the maximally powerful e-value $f_Q / f_{\tilde Q}$.
\end{proof}

We remark that both $\tilde Q$ and $\tilde Q^a$ are $\argmax_{P \in \cM^-} \left\{ \sum f_Q(n) \log {f_P(n)} \right\}$. However, \citet[Definition 3.4]{numeraire} require that \emph{the} RIPr be absolutely continuous with respect to $Q$. Therefore, only $\tilde Q^a$ satisfies the definition. Effectively, changing the measure values on $Q$-null numbers is immaterial to either the entropy minimization or the e-power maximization. We call $\tilde Q^a$ instead of $\tilde Q$ the RIPr for congruence with the \cite{numeraire} formulation.

We also remark that when the support of $Q$ is finite, the property of RIPr \eqref{eqn:log-comparison-rn} here implies the well-known fact in shape-constrained inference that the Grenander density estimator is the maximum likelihood estimator \citep{grenander1981abstract,jankowski2009estimation} --- to see that, set $Q$ to be the empirical measure. We shall additionally prove a continuous version of \cref{thm:ripr-lcm} in \cref{sec:cont-ripr}, and discuss its relation to numerous existing results in the continuous setting by e.g.\ \cite{patilea2001convex,jankowski2014mis,samworth2025nonparametric}.



Finally, we can extend \cref{thm:ripr-lcm} to an arbitrary number of i.i.d.\ observations. Denote by $\cM^n_\iids$ the set of all $n$-fold product measures of some monotone distribution: $\cM^n_\iids = \{ P^n : P \in \cM \}$. The log-optimal e-value and the RIPr are both the $n$-fold product of the respective one-observation quantity.

\begin{corollary}\label[corollary]{cor:optimal-eprocess} \label{thm:ripr-lcm-n} Let $\tilde Q$ be the least concave majorant of $Q$. Then, $(\tilde Q^{a})^n$ is the reverse information projection of $Q^n$ onto the set $\cM^n_\iids $; and the following e-value for $\cM^n_\iids$ 
\begin{equation}\label{eqn:numeraire-nsm}
        E^*_Q(x_1,\dots, x_n) = \prod_{k=1}^n \frac{f_Q (x_k) }{f_{\tilde Q} (x_k)}
    \end{equation}
    has the largest
  e-power under $Q^n$,
    where $0/0$ is understood as 0. Additionally,
   letting $\tilde Q^a$ be the absolutely continuous part of $\tilde Q$ with respect to $Q$, the \emph{reverse information projection} of $Q^n$ onto $\cM^n_\iids$ is $(\tilde Q^a)^n$.
\end{corollary}

\begin{proof}
    Note that in the one-observation case (\cref{thm:ripr-lcm}), the RIPr $\tilde Q^a$ is upper  
    bounded by $\tilde Q \in \cM$. We invoke \cref{prop:ripr-bounded-nobs} to see that the RIPr is tensorizable in this case, i.e.\ $(\tilde Q^a)^n$ is the RIPr of $Q^n$ onto $\cM^n_\iids$. Consequently, the numeraire e-value is
    \begin{equation}
       \frac{f_{Q^n}(x_1,\dots, x_n)}{f_{(\tilde Q^a)^n}(x_1,\dots, x_n)} = \prod_{k=1}^n \frac{f_Q (x_k) }{f_{\tilde Q^a} (x_k)}  = \prod_{k=1}^n \frac{f_Q (x_k) }{f_{\tilde Q} (x_k)}. \qedhere
    \end{equation}
\end{proof}

In fact, \eqref{eqn:numeraire-nsm} is a supermartingale under $\cM^\infty_\iids$ (i.e.\ under any infinite sequence of i.i.d.\ observations $X_1,X_2,\dots$ distributed according to some $P \in \cM$), by \cref{prop:ripr-bounded-nobs} it is also the log-optimal e-value, e-process, and nonnegative supermartingale under $X_1,X_2,\dots \iid Q$ at arbitrary fixed sample size or bounded stopping time.

\section{Some extensions to continuous distributions}\label{sec:cont}

Our exposition thus far has carved out a near-complete landscape for the testing problem of discrete unimodality. Extending these results when the support of the distributions goes from $\mathbb Z$ to $\mathbb R$ is sometimes straightforward. In particular, we are able to generalize numerous results in the one-observation regime (\cref{sec:1obs}).
We record some of these extensions in this section while leaving a more involved study (including sequential tests) for future work.  If $P \in \cS(\mathbb R)$, we denote by $F_{P}$ the cumulative distribution function (CDF) of $P$, i.e.\ $F_P(x) = P((-\infty, x])$.

We start by defining unimodality on $\mathbb R$. 
Let us first follow \cite{basudasg} and define unimodality by the convexity of the cumulative distribution function (CDF) without requiring a PDF, thus allowing a jump at the mode. We shall discuss the jumpless case later.
Below are some sets of distributions of interest.

\begin{definition}
    Let $\cU^+_{\theta}$ be the set of distributions on $[\theta, \infty)$ with concave CDF.  Let $\cU^-_{\theta}$ be the set of distributions on $(-\infty, \theta]$ with convex CDF. Finally, we let $\cU_{\theta}$ be the set of distributions on $\mathbb R$ whose CDF is concave on $[\theta, \infty)$ and convex on $(-\infty, \theta]$. We call $\cU^+_{0}$ the set of monotone distributions, and $\cU_\theta$ the set of $\theta$-unimodal distributions.
\end{definition}

\subsection{Power-one e-process for $\cU_0^+$ against $\cU_0^{+c}$}

Let us first generalize \cref{thm:power1-monotone} by constructing a power-one e-process for continuous monotonicity. The key idea is the \emph{rational} one-observation witness e-values of non-monotone distributions, which enable a countable mixture. First, let us show that it is always possible to witness a non-monotone distribution on $[0,\infty)$ at rational points.

\begin{lemma}
    For any non-monotone distribution $Q \in \cP(\mathbb R^{\ge 0 }) \setminus \cU_0^+$, there exist \emph{rational} numbers $0 \le a < b < c$ such that
    \begin{equation}\label{eqn:witness-rationals}
        \frac{F_Q(b) - F_Q(a)}{b-a} <  \frac{F_Q(c) - F_Q(b)}{c-b}.
    \end{equation}
\end{lemma}
\begin{proof}
    Suppose on the contrary that for any rationals $0 \le a < b < c$,
    \begin{equation}\label{eqn:contrapositive-rationals}
        \frac{F_Q(b) - F_Q(a)}{b-a} \ge  \frac{F_Q(c) - F_Q(b)}{c-b}. 
    \end{equation}
    Then we can see that \eqref{eqn:contrapositive-rationals} holds for any real numbers $0 \le a < b < c$ as well by right-approximating the right-continuous function $F_Q$. This contrapositively proves the lemma.
\end{proof}

We are now able to construct a witness e-value based on these rational points, which is completely similar to \cref{lem:log-power-monotone-evalue}.

\begin{lemma}\label[lemma]{lem:log-power-monotone-evalue-continuous}
     Let $Q \in \cP(\mathbb R^{\ge 0 }) \setminus \cU_0^+$ with rationals $0 \le a < b < c$ satisfying \eqref{eqn:witness-rationals}.
   Then, the function $E^{Q,a,b,c} \in \fF(\mathbb R^{\ge 0})$ defined as
    \begin{equation}
    E^{Q,a,b,c}(x) = 1 + \lambda_{Q,a,b,c} \left( -  \frac{\id_{\{a < x \le b \} }}{b-a} +  \frac{\id_{\{b < x \le c \} }}{c-b} \right)
\end{equation}
where
\begin{equation}
    \lambda_{Q,a,b,c} = \frac{
  \frac{F_Q(c) - F_Q(b)}{c-b} - \frac{F_Q(b) - F_Q(a)}{b-a}
    }{2\left(  \frac{F_Q(b) - F_Q(a)}{b-a} +  \frac{F_Q(c) - F_Q(b)}{c-b}  \right)} \wedge (b-a)
\end{equation}
satisfies 
\begin{equation}
   E^{Q,a,b,c} \in \cU_0^{+\circ} \quad \text{and} \quad \Expw_{X \sim Q} \log E^{Q,a,b,c}(X)    > 0.
\end{equation}
\end{lemma}

The proof of \cref{lem:log-power-monotone-evalue-continuous} is similar to that of \cref{lem:log-power-monotone-evalue}. Now, a power-one e-process with unknown $Q$ but known rational witness points $a,b,c$, analogous to \cref{lem:oracle-power1}.

\begin{lemma}
     Let $Q \in \cP(\mathbb R^{\ge 0 }) \setminus \cU_0^+$ with rationals $0 \le a < b < c$ satisfying \eqref{eqn:witness-rationals}. Then, the process
     \begin{equation}
         M_n^{(a,b,c)} = \prod_{k=1}^n E^{  \hat Q_{k-1} ,a,b,c}(X_k)
     \end{equation}
    diverges to $\infty$ exponentially fast almost surely under the alternative $X_1,X_2 \dots \iid Q$:
    \begin{equation}
       \liminf_{n \to \infty}  \frac{\log  M_n^{(a,b,c)} }{n}  > 0;
    \end{equation}
    and is a supermartingale under the null $X_1,X_2 \dots \iid P \in \cU_0^+$.
\end{lemma}
To see this, examine the proof of \cref{lem:oracle-power1} in \cref{sec:pf-orcl}: since $\lambda_{\hat Q_{k-1},a,b,c} \stackrel{a.s.}\to \lambda_{Q,a,b,c}$ still holds, the same martingale law of large number argument applies. Finally, we apply a mixture over the rational triples $(a,b,c)$ to arrive at the desired power-one test against aribitrary unknown $Q \in  \cP(\mathbb R^{\ge 0 })\setminus \cU_0^+$.

\begin{theorem}[Power-one sequential test of continuous monotonicity]
     Let $Q \in  \cP(\mathbb R^{\ge 0 })\setminus \cU_0^+$. Let $\rho$ be a probability measure on $\{ (a,b,c)\in \mathbb Q^3 : 0 \le a < b < c \}$ such that $\rho(a,b,c) = \rho(\{ (a,b,c) \}) > 0$ for all $a,b,c$ (which exists due to the countability of such rational triples $(a,b,c)$).
     Then, the process
     \begin{equation}
      M_n = \sum_{a,b,c} \rho(a,b,c) M_n^{(a,b,c)}
     \end{equation}
     diverges to $\infty$ exponentially fast almost surely under the alternative $X_1,X_2 \dots \iid Q$:
    \begin{equation}
       \liminf_{n \to \infty}  \frac{\log  M_n }{n}  > 0;
    \end{equation}
    and is a supermartingale under the null $X_1,X_2 \dots \iid P \in \cU_0^+$.
\end{theorem}

It is easy to see that a power-one unimodality test for $\cU_0$ against $\cU_0^c$ can then be constructed via $\sum_{a,b,c} \rho(a,b,c) \frac{M_n^{(a,b,c)} +  M_n^{(-a,-b,-c)}}{2}$, where $ M_n^{(-a,-b,-c)}$ is the power-one supermartingale testing $\cU_0^{-}$ with witnesses $-c < -b < -a \le 0$. A power-one unimodality test for unknown mode, however, is left to future work.

\subsection{Basic inequality and the $Xq(X)$ e-values}

We discussed in \cref{sec:x+1qx} that functions of the form $(X+1)q(X)$ are e-values for discrete monotonicity (and consequently unimodality). It is easy to see that if the support is $\gamma \mathbb Z$ instead of $\mathbb Z$, one needs to use $(X+\gamma)q(X)$ instead to match the underlying granularity. Letting the granularity $\gamma \to 0$, one would speculate that $Xq(X)$ leads to e-values for $\mathbb R$-valued monotone and unimodal distributions. We formally show this in this subsection.


Continuous monotone distributions satisfy the following inequality,  which we phrase as a set containment statement.

\begin{lemma}[Basic inequality for continuous monotone distributions]\label[lemma]{lem:basic} Define the set
\begin{equation}
    \cV = \left\{ P \in \cP([0,\infty)) :  P((a,b]) \le \frac{b-a}{b} \text{ for all } 0 \le a < b \right\}
\end{equation}
Then, $\cU^+_0 \subseteq \cV$.
\end{lemma}
\begin{proof} Let $P \in \cU^+_0$.
Since the CDF of $P$ is concave on $[0,\infty)$, we have, as $a/b \in[0,1]$,
\begin{equation}
    (1- a/b) F_P(0) + (a/b) F_P(b) \le F_P(a).
\end{equation}
Therefore, we see
that
\begin{equation}
    F_P(b) - F_P(a) \le (1 - a/b)(F_P(b) - F_P(0)) \le 1 - a/b. \qedhere
\end{equation}
\end{proof}
There is a clear analogy between the continuous distribution class $\cV$ and the discrete distribution class $\cM'$ introduced earlier at \eqref{eqn:discrete-bi} in \cref{sec:x+1qx}. To justify the name ``basic inequality'' for both $\cV$ and $\cM'$, note that \cref{lem:basic} is related to numerous
variations of it that appear in the existing unimodal literature. The celebrated result by \cite{Edelman01111990}, for example, can be derived from \cref{lem:basic}, which we discuss in \cref{sec:rat}. In their recent work, \cite{paul2025finite} call the inequality $
    P(I_1)/|I_1| > P(I_2)/|I_2| $
``the fundamental property of unimodal distribution functions'', where $I_1 = (a,b]$ and $I_2=(b,c]$ are on the right side of the mode. It is clear how the proof of \cref{lem:basic} above is also related to this inequality.

Again, $\cU_0^+$ is a proper subset of $\cV$ with the following example. For $\lambda \ge 0$ define $Q_{\lambda}$ to be the distribution that is a 60-40 mixture of uniform-$[0,1]$ and uniform-$[1+\lambda, 2+\lambda]$. $Q_{\lambda} \in \cU^+_0$ only when $\lambda = 0$. As can be seen from \cref{tab:cont-example}, if $\lambda \in (0, 0.5]$, $Q_\lambda$ belongs to $\cV$ but not $\cU_0^{+}$.

\begin{table}[!t]
    \centering
    \begin{tabular}{|c|c|c|c|c|} \hline
 $ \sup_{a<b} \frac{Q_{\lambda}((a,b]) b}{b-a} $  & $a\in [0,1]$ & $a\in [1,1+\lambda]$ &  $a\in[1+\lambda, 2+\lambda]$ & $a\in[2+\lambda, \infty)$  \\ \hline
  $b \in [0,1] $  & $0.6$ & / & / & / \\ \hline
  $b \in [1,1+\lambda] $  & $0.6$ & 0 & / & / \\ \hline
  $b \in [1+\lambda, 2+\lambda] $  & 1 & $0.4(2+\lambda)$ & $0.4(2+\lambda)$ & / \\ \hline
  $b \in [2+\lambda, \infty) $  & 1 & $0.4(2+\lambda)$ & $0.4(2+\lambda)$ & 0 \\  \hline
\end{tabular}
    \caption{Suprema of $\frac{Q_{\lambda}((a,b]) b}{b-a}$ on different ranges of $a < b$.}
    \label{tab:cont-example}
\end{table}

The following e-value follows directly from \Cref{lem:basic}.
\begin{lemma}[Bump e-value] \label[lemma]{lem:bump}
    For any $b > a > 0$, the random variable 
    \begin{equation}
        L_{a,b}(X) =  \frac{b}{b-a} \id_{\{ a < X \le  b \}}
    \end{equation}
    is an e-value for $\cV$, consequently for $\cU^+_0$.
\end{lemma}

We now employ the method of mixture, which previously appeared in constructing omni-powerful sequential tests in \cref{sec:seq}, on a family of well-chosen bump e-values. The result of this mixture is a continuous version of the $(X+1)q(X)$ e-values discussed in \cref{sec:x+1qx}.

\begin{theorem}\label{thm:cont}
    Let $q$ be a left-continuous PDF on $[0,\infty)$ such that $(x+1)q(x)$ is bounded. Then, $X q(X)$ is an e-value for $\cV$,  consequently for $\cU^+_0$.
\end{theorem}
The high-level idea of the proof is that we (1) fix the ``bandwidth'' $w = b-a$, (2) put a $q$-mixture on $a$,  and (3) take the $w\to 0$ limit. The full proof can be found in \cref{sec:pf-cont}. These $Xq(X)$ e-values, like their discrete siblings (\cref{cor:test-Z,cor:1obs-CI}), also lead to a finite-length confidence interval for the mode $\theta$. However, this CI is strictly longer (although only slightly so) than the CI of \cite{Edelman01111990}, so we leave it as \cref{cor:1obs-CI-R} in \cref{sec:pf-cont} after stating the proof of \cref{thm:cont}.

\subsection{Polar of $\cU_0^+$}

Recall that we constructed the polar of the set of discrete monotone distribution class $\cM$ as \eqref{eqn:set-M}, which can be equivalently expressed as
\begin{equation}\label{eqn:Mcirc-reprise}
  \cM^\circ =  \fM = \left\{ E \in \fF(\mathbb Z^{\ge 0}) : \  \sum_{k=0}^n E(k) \le n +1  \text{ for all } n \right\}.
\end{equation}
A conceptually natural generalization of the class $\fM$ in the continuous regime is the following.
\begin{equation}
    \fU = \left\{ E \in L_1([0,\infty)) : E \ge 0, \;  \int_0^x  E(t) \d t \le x \text{ for all } x \right\}.
\end{equation}
It is worth remarking here that $\fU$ describes the condition that the function $E$ is an e-value for all \emph{uniform} distributions on $[0,x]$ for $x > 0$.

Let us use a simple integration-by-parts argument, which is the continuous counterpart of the summation-by-parts argument used frequently in the discrete case, to investigate whether the functions belonging to $\fU$ are e-values for distributions in $\cU_0^+$. 

To this end, first recall the smoothness of concave CDFs (see e.g.\ \citet[Section 2]{cipra1978class}). Let $P \in \cU_0^+$, and $F_P$ be its concave CDF on $[0, \infty)$.  Then, it has a unique Lebesgue decomposition w.r.t.\ the Lebesgue measure on $[0,\infty)$:
\begin{equation}\label{eqn:decompose}
     P(\d x) = f_P(x) \d x + P(\{0\}) \delta_0( \d x) 
\end{equation}
  where $f_P(x)$ is the non-increasing derivative of $F_P$ which exists almost everywhere on $(0, \infty)$, as the convex function $F_P$ is absolutely continuous on any interval $[a, b] \subset (0, \infty)$. The only possible singular part is a jump at the boundary $x=0$, represented by the Dirac mass $\delta_0$ in \eqref{eqn:decompose}.
  Any expected value under $P$ is representable by:
  \begin{equation}
       \Expw_{X \sim P} g(X) = g(0) \cdot P(\{0\}) + \int_0^\infty g(x) f_P(x) \d x.
  \end{equation}
  Additionally, noting that
  \begin{equation}\label{eqn:xfx0}
      x f_P(x) =  \int_0^x f_P(x)  \d t \le  \int_0^x f_P(t) \d t \stackrel{ x \to 0 }{\longrightarrow} 0,
  \end{equation}
  we see that $x f_P(x) \to 0$ as $x \to 0$.

We are now ready to state and prove the following continuous counterpart of \cref{thm:polar-of-M}, for which let us define
\begin{equation}
    \cU_{00}^+ = \{ P \in \cU_{0}^+  : P(\{0\}) = 0 \},
\end{equation}
the ``jumpless continuous monotone'' class in which every distribution has a non-increasing density on $(0,\infty)$. Polars of both $\cU_0^+$ and $\cU_{00}^+$ are provided below.

\begin{theorem}\label{thm:polar-of-U}
    $\cU_{00}^{+\circ} = \fU$. That is, functions in $\fU$ are exactly all e-values testing \underline{jumpless} continuous monotonicity. Further, $\cU_{0}^{+\circ} = \fU_1 := \{ E \in \fU : E(0) \le 1 \} $. That is, functions in $\fU$ with $E(0)\le 1$ are exactly all e-values testing continuous monotonicity with a possible jump at 0.
\end{theorem}
\begin{proof}
    First, suppose $E \in \fU$. Let $r(x) =  \int_0^x  E(t) \d t$, and $P \in \cU_{0}^+$.  For any $0 < a< b < \infty$, consider the following integration by parts of the Riemann–Stieltjes integral:
    \begin{equation}
        \int_a^b f_P(x) \d (r(x)-x) =  (r(x)-x)f_P(x) |_a^b - \int_a^b (r(x) - x) \d f_P(x),
    \end{equation}
    where the existence of the Riemann–Stieltjes integral on the right hand side (and hence the fact that the equality holds) is ensured by the continuity of $r(x) -x $ and the monotonicity of $f_P$. Further, since $r(x) - x \le 0$ and $f_P(x)$ is non-increasing, we have $\int_a^b (r(x) - x) \d f_P(x) \ge 0$. Noting that $r(b) -b \le 0$ and $r(a) \ge 0$, we see that
    \begin{equation}
        \int_a^b f_P(x) \d (r(x)-x) \le  a f_P(a).
    \end{equation}
    We have established around \eqref{eqn:xfx0} that $\lim_{a \to 0} a f_P(a) = 0$. So
\begin{equation}
        \int_0^\infty f_P(x) \d (r(x)-x) \le 0.
    \end{equation}
    Now, we can compute the expected value:
    \begin{align}
  &  \Expw_{X \sim P} E(X) = E(0)P(\{0\}) + \int_0^\infty E(x) f_P(x) \d x 
   \\
   = & E(0)P(\{0\}) + \int_0^\infty  f_P(x) \d (r(x)-x) + \int_0^\infty f_P(x) \d x
    \\
    \le & E(0)P(\{0\}) + P((0,\infty)) \stackrel{(*)}\le 1,
\end{align}
where $(*)$ holds as long as $P(\{0\}) = 0$ (in which case $P \in \cU_{00}^+$) or $E(0) \le 1$ (in which case $E \in \fU_1$).
Therefore, 
$\fU \subseteq \cU_{00}^{+\circ}$ and $\fU_1 \subseteq \cU_0^{+\circ}$.

Second, suppose $E \in \cU_{00}^{+\circ}$. The uniform measure on $[0,x]$ under which $E$ is an e-value immediately implies $\int_0^x E(t) \d t \le x$. So $E \in \fU$, implying that $\cU_{00}^{+\circ} \subseteq \fU$. If $E \in \cU_{0}^{+\circ}$, we can additionally consider the point mass at 0 to conclude that $E(0) \le 1$. So $E \in \fU_1$, implying that $\cU_{0}^{+\circ} \subseteq \fU_1$.

We have now proved that $\fU = \cU_{00}^{+\circ}$ and $\fU_1 = \cU_0^{+\circ}$.
\end{proof}

We can also immediately see that ${^\circ_1}\fU = \cU_{00}^{+}$ and ${^\circ_1}\fU_1 = \cU_{0}^{+}$ by noting that if $Q \notin \cU_{0}^{+}$ with violation points $a < b < c$:
\begin{equation}
    \frac{F_Q(b) - F_Q(a)}{b-a} <  \frac{F_Q(c) - F_Q(b)}{c-b},
\end{equation}
then the function
\begin{equation}\label{eqn:witness-cont}
    f(x) = 1  + \frac{b-a}{c-b}\id_{\{b < x \le c \}}  -   \id_{\{a < x \le b\}} \ge 0
\end{equation}
belongs to $\fU_1$ but is not an e-value under $Q$. It remains to be verified, however, if a full bipolar statement similar to \cref{thm:bipolar-monotone} can be established in this case as well. Nonetheless, our argument herein already establishes the similar testability of continuous monotonicity: for any non-monotone \emph{probability} measure $Q$, a witness $\cU_0^+$-e-value \eqref{eqn:witness-cont} always exists.

We can similarly show that the set of all e-values for the jumpless continuous unimodal class $\cU_{00}$ is $\{ E \in L_1(\mathbb R) : E(x), E(-x) \in \fU \}$, and the set of all e-values for the continuous unimodal class $\cU_{0}$ is $\{ E \in L_1(\mathbb R) : E(x), E(-x) \in \fU_1 \}$. We omit these straightforward proofs.

\subsection{Reverse information projection, continuous}\label{sec:cont-ripr}


For continuous random variables, we have the following generalization of \cref{thm:ripr-lcm} on the log-optimal e-value for $\cU_0^+$ under an alternative $Q$ as well as the reverse information projection of $Q$ onto $\cU_0^+$. 

We denote by $\tilde Q \in \cU_0^+$ the least concave majorant of $Q \in \cP([0,\infty))$. As an additional regularity condition, we assume that $Q \ll \cU_0^+$, that is,
\begin{equation}
   (\forall P \in  \cU_0^+,  P(A) = 0) \implies  Q(A) = 0,
\end{equation}
a central condition to establish the numeraire-RIPr results by \cite{larsson2025variables}. This requires $Q$ to have the same decomposition as \eqref{eqn:decompose}:
\begin{equation}
     Q(\d x) = f_Q(x) \d x + Q(\{0\}) \delta_0( \d x) 
\end{equation}
i.e.\ $Q$ is absolutely continuous on $(0,\infty)$ with only a possible jump at 0.

\begin{theorem}
    For any $Q \in \cP([0,\infty))$ such that $Q \ll  \cU_0^+$, it holds that $Q \ll \tilde Q$ and $\frac{\d Q}{\d \tilde Q} \in  \fU_1$ where $\frac{\d Q}{\d \tilde Q}$ is any Radon-Nikodym derivative between $Q$ and $\tilde Q$ such that $\frac{\d Q}{\d \tilde Q}(x) \in [0,1]$ on $\{x : f_{\tilde Q}(x) = f_{Q}(x) = 0\}$. Consequently, letting $\tilde Q^a$ be the absolutely continuous part of $\tilde Q$ with respect to $Q$:
    \begin{itemize}
        \item  $\tilde Q^a$ is the reverse information projection of $Q$ onto the set $\cU_0^+$;
        \item the ratio
        \begin{equation}
        \frac{\d Q}{\d \tilde Q} = \frac{\d Q}{\d \tilde Q^a} \quad \text{(in the $Q$-a.s.\ sense)}
    \end{equation}
    is the numeraire (log-optimal) e-value for $\cU_0^+$ under $Q$; and
    \item for any non-random $n$, $(\tilde Q^{a})^n$ is the reverse information projection of $Q^n$ onto the set $\cU_{0\,\iids}^{+n}$, and the product \begin{equation}
       M_n =  \prod_{i=1}^n  \frac{\d Q}{\d \tilde Q}(X_i)
    \end{equation}
     forms the log-optimal e-value, nonnegative supermartingale, and e-process for $\cU_0^+$ under $X_1, \dots, X_n \iid Q$; additionally, $M_\tau$ is the log-optimal e-value at any finite stopping time $\tau$.
    \end{itemize}

\end{theorem}
Indeed, the Radon-Nikodym derivative is only defined up to a $\tilde Q$-negligible set and it can take arbitrary value on $\{x : f_{\tilde Q}(x) = f_{Q}(x) = 0\}$. We must stipulate that it takes values in $[0,1]$ there.
\begin{proof}
    At 0, both $Q$ and $\tilde Q$ can possibly have an atom and $\tilde Q(\{0\}) \ge Q(\{0\})$ by the definition of LCM. On $(0,\infty)$, per the basic property of the LCM,
the domain $(0, x)$ can be partitioned into the contact set $S_{\text{eq}} = \{x : F_{\tilde Q}(x) = F_Q(x)\}$ where $f_Q (x)= f_{\tilde Q}(x)$ holds; and a union of open intervals $S_{\text{flat}} = \bigcup_i (a_i, b_i)$ where $	\hat F_Q(x) > F_Q(x)$, where $F_{\tilde Q}$ is linear with constant slope $f_{\tilde Q}(x) = c_i > 0$. Therefore $f_{\tilde Q}(x) = 0$ always implies $f_Q(x) = 0$. We thus see that $Q \ll \tilde Q$ and that  $\frac{\d Q}{\d \tilde Q}(0) \le 1$. Now consider the integral
    \begin{equation}
     I_x =   \int_0^x \frac{\d Q}{\d \tilde Q}(t) \d t = \int_{(0, x) \cap S_{\text{eq}}} \frac{\d Q}{\d \tilde Q}(t) \d t + \int_{(0, x) \cap S_{\text{flat}}} \frac{\d Q}{\d \tilde Q}(t) \d t.
    \end{equation}
    The first integral above is at most $|(0, x) \cap S_{\text{eq}}|$ since $\frac{\d Q}{\d \tilde Q}(t) \in [0,1]$ on $S_{\text{eq}}$. The second integral above can be further decomposed into pieces:
    \begin{align}
      &  \int_{(0, x) \cap S_{\text{flat}}} \frac{\d Q}{\d \tilde Q}(t) \d t = \sum_i \int_{a_i}^{b_i} \frac{f_Q(x)}{c_i} \d x = \\ & \sum_i  \frac{F_Q(b_i) - F_Q(a_i)}{F_{\tilde Q}(b_i) - F_{\tilde Q}(a_i)}(b_i - a_i) =  \sum_i (b_i - a_i) = | (0, x) \cap S_{\text{flat}} |.
    \end{align}
    Therefore we see that $I_x \le |(0,x)| = x$, and that consequently $\frac{\d Q}{\d \tilde Q} \in \fU_1$.

    The rest of the theorem follows from the basic result in \cite{numeraire} (which we review as \cref{prop:fundamental-numeraire}), and the tensorization property \cref{prop:ripr-bounded-nobs}.
\end{proof}

Finally, we remark that it is well known in the shape-constrained inference literature that the (left derivative of) least concave majorant is the projection of a density onto the space of monotone densities. For example, to quote from \citet[Theorem 2.2]{samworth2025nonparametric}, if a density $f_Q$ on $(0,\infty)$ satisfies $\int f_Q(x) |\log (f_Q(x))| \d x < \infty$ and $\int f_Q(x) |\log x| \d x < \infty$, then the LCM's density $f_{\tilde Q}$ satisfies
\begin{equation}\label{eqn:lcm-kl-proj}
    f_{\tilde Q} = \argmin_{f} \kl(f_Q \| f)
\end{equation}
where the argmin is taken over all monotone densities $f$.
From the properties of the RIPr stated before \cref{thm:ripr-lcm}, we see that the RIPr statement on $\tilde Q^a$ generalizes this projection statement by ensuring the \emph{relative} KL minimization
holds:
$\tilde Q^a$ is the unique measure in the effective null ${^\circ}(\cU_{0}^{+\circ}) = {^\circ}\fU_1$ such that
\begin{equation}\label{eqn:lcm-rel-kl}
    \Expw_{X \sim Q} \left\{ \log \frac{\d P^a}{\d \tilde Q^a}(X) \right\} \in [-\infty, 0] \quad \text{for all $P \in {^\circ}(\cU_{0}^{+\circ}) $}.
\end{equation}
Recall that this relative version of the KL projection is finer than \eqref{eqn:lcm-kl-proj} and may distinguish among measures with infinite KL divergence to $Q$. The relative KL minimization statement coincides with Theorem 2.1 in \cite{jankowski2014mis} (which the author quotes from \cite{patilea2001convex}), with the only minor difference being that our result allows a jump at 0. See also Remark 2.2 in \cite{jankowski2014mis} for a discussion related to \eqref{eqn:lcm-rel-kl} implying \eqref{eqn:lcm-kl-proj}.

\section{Miscellaneous technical discussions}\label{sec:misc}

\subsection{On the p-value and CI of \cite{Edelman01111990}}\label{sec:rat}

\cite{Edelman01111990} derives a one-observation confidence interval (weak, by \cref{def:all-modes-CI}) for mode of a unimodal distribution. In this section, we restate his result through the ``basic inequality'' we described in \cref{lem:basic}, offering a point of comparison for various other results in this paper. Recall that the set of all distributions satisfying the basic inequality on $[0,\infty)$
\begin{equation}\label{eqn:bi-appendix}
    P((a,b]) \le \frac{b-a}{b} \text{ for all } 0 \le a < b
\end{equation}
is denoted by $\cV$,
which contains $\cU^+_0$, the set of all distributions with concave CDF on $[0,\infty)$.

A p-value for testing the null hypothesis $\cP$ is defined as a nonnegative statistic $p(X)$ such that $\Pr_{X \sim P}(p(X) \le u) \le u$ for all $u\in[0,1]$ under any null distribution $P \in \cP$.

\begin{proposition}[\citeauthor{Edelman01111990}'s p-value]
    If $a \ge 0$, the random variable
    \begin{equation}
        R_{a}^{\mathsf{E}}(X) = \frac{2|X-a|}{|X-a| + |X|}
    \end{equation}
    is a p-value for $\cV$, therefore for $\cU_0^+$; if $a \le 0$, a p-value for $\cU_0^{-}$; consequently, for any $a \in \mathbb R$, a p-value for $\cU_0$.
\end{proposition}
\begin{proof}
    We only need to prove the case for $a \ge 0$. For any $P \in \cV$ and $u \in (0,1)$,
    \begin{equation}
        P( R_a^{\mathsf{E}}(X) \le u ) = P\left(  \left[ \frac{2-u}{2}a, \frac{2-u}{2-2u} a \right] \right) \le \frac{ \frac{2-u}{2-2u} - \frac{2-u}{2}}{ \frac{2-u}{2-2u}}  = u,
    \end{equation}
    where the $\le$ is obtained directly by applying the basic inequality \eqref{eqn:bi-appendix}.
\end{proof}

Applying the usual ``reject when $p < 0.05$'' rule leads to the confidence set.

\begin{proposition}[\citeauthor{Edelman01111990}'s CI] \label[proposition]{prop:edelci}
     Define  $L_\alpha = 2/\alpha - 1$. For any $\phi \in \mathbb R$,  the set
     \begin{equation}
              ( X - L_\alpha |X - \phi|, X + L_\alpha |X -  \phi| )
    \end{equation}
     is a weak $(1-\alpha)$-confidence interval for mode. 
\end{proposition}
\begin{proof}
    We reject the null ``$\text{mode} = \theta$'' if the p-value $R_{\theta - \phi}^{\mathsf{E}}(X-\theta) \le \alpha$. The unrejected $\theta$'s form a $(1-\alpha)$-CI:
    \begin{equation}
       \left\{ \theta :   \frac{2|X-\phi|}{|X-\phi| + |X-\theta|}  > \alpha \right\} 
    \end{equation}
    which equals $( X - L_\alpha|X - \phi|, X + L_\alpha|X - \phi| )$.
\end{proof}

\subsection{Behind-the-scenes: constrained hypotheses}\label{sec:constraints}

We now come to the question of \emph{how} these polar sets, $\fM = \cM^\circ$ in \cref{thm:polar-of-M} and $\fD_{\theta} = \cD_{\theta}^\circ$ in \cref{thm:polar-unimodal-Z}, were discovered. This key to finding these polars is that the distribution classes $\cM$ and $\fD_{\theta}$ can both be generated by \emph{constraints}, in the language of \cite{larsson2025variables}. If $\fG$ is a set of measurable functions $\mathbb A \to \mathbb R$, we define the set of distributions these functions as constraints generate as
\begin{equation}
       \cP(\mathbb A; \fG) = \left\{ P  \in \cP(\mathbb A) :  \int |f| \d P < \infty \text{ and }  \int f \d P \le 0, \; \forall f \in \fG \right\}. 
\end{equation}
Then, it can be seen that any function $h$ satisfying
\begin{equation}
    h \ge 0, \quad h(x) = 1 + \sum_{f \in \fH} \pi_f f(x)
\end{equation}
where $\{\pi_f\}_{f\in \fH}$ are nonnegative constants and $\fH$ is a finite subset of $\fG$,
would satisfy $h \in  \cP(\mathbb A; \mathbf G)^\circ$. In fact, it is proved by \citet[Theorem 3.1]{larsson2025variables} that if $\fG$ is finite the set $\cP(\mathbb A; \mathbf G)^\circ$ (i.e.\ \emph{all} e-values for $\cP(\mathbb A; \fG)$) contains exactly those nonnegative functions that are bounded by some
\begin{equation}\label{eqn:combination}
  1 + \sum_{f \in \fG} \pi_f f(x), \quad \pi_f \ge 0.
\end{equation}
The case for infinite $\fG$ is harder to obtain, and in general admits no closed-form expression. However, as we shall show soon, while the classes $\cM$ and $\cD_{\theta}$ are both generated by some infinite set of constraint functions, their polars $\cM^\circ$ and $\cD_\theta^\circ$ coincide with those functions \eqref{eqn:combination} where the infinite sum $\sum_{f \in \fG}$ is actually finite.

\begin{fact}\normalfont For $k \in \mathbb Z^{\ge 0}$,
    let $h_k(n)  = -\id_{\{n=k\}} + \id_{\{n=k+1\}}$ and let $\mathbf G_M = \{ h_0, h_1,\dots \}$. Then,
    \begin{equation}
        \cM = \cP( \mathbb Z^{\ge 0}; \mathbf G_M);
    \end{equation}
    and the set of nonnegative functions of the following form
    \begin{equation}\label{eqn:E-GM}
         E(n) = 1 + \sum_{k=0}^\infty \pi_{k+1} h_k(n) = 1 + \pi_n - \pi_{n+1}
    \end{equation}
    coincides with $ \cM^\circ = \fM$ defined in \eqref{eqn:set-M}, which can can be seen from the reparametrization $\rho_n = n - \pi_n$.
\end{fact}

\begin{fact}\normalfont  For $k \in \mathbb Z^{\ge 0}$,
    still defining $h_k(n)  = -\id_{\{n=k\}} + \id_{\{n=k+1\}}$, but now define $\mathbf G_D = \{ h_0(n), h_1(n),\dots ;  h_0(-n), h_1(-n),\dots \}$. Then,
    \begin{equation}
        \cD_0 = \cP( \mathbb Z; \mathbf G_D).
    \end{equation}
    and the set of nonnegative functions of the following form
    \begin{equation}\label{eqn:E-GD}
         E(n) = 1 + \sum_{k=0}^\infty \pi_{k+1} h_k(n) +  \sum_{k=0}^\infty \sigma_{k+1} h_k(-n) = \begin{cases}
             1 + \pi_n - \pi_{n+1} & n \ge 1 \\
             1 - \pi_1 - \sigma_1 & n = 0 \\
             1 + \sigma_{-n} - \sigma_{-n+1} & n \le -1
         \end{cases}
    \end{equation}
     coincides with $\cD_0 ^\circ = \fD_0$ defined in \eqref{eqn:set-Dtheta}, which can be seen from the reparametrization $\rho_n = n - \pi_n$ and $\eta_n = n-\sigma_n$.
\end{fact}

It is important to remark here that the distribution classes $\cM$ and $\fD_{\theta}$ being constraint-generated does \emph{not} directly imply that their polar sets are exactly expressible via the constraints. The constraints only act here as a heuristic for us to make educated guesses of the polars, which we verify ``manually''.
The quoted result by \citet[Theorem 3.1]{larsson2025variables}, again, only holds for finitely generated hypotheses. The hypotheses $\cM$ and $\fD_{\theta}$ are countably generated, although through the expressions \eqref{eqn:E-GM} and \eqref{eqn:E-GD} it can be seen that the formally infinite sums are always well-defined finite sums (without the need to invoke any topology) within these constraints. We speculate that this is related to why the polars are exactly the ones we construct from the constraints, and we leave the formal treatment of this idea for future work.

\subsection{Tensorization of reverse information projection}\label{sec:tensorization}

We have extensively utilized the concepts of the numeraire and the reverse information projection developed by \cite{numeraire} in \cref{sec:ripr,sec:cont-ripr}, where we alluded to the fact that these one-observation results sometiems tensorize into the multiple-observation space.
In this section, we review the key results by \cite{numeraire} necessary for such multiple-observation generalizations. Throughout the section, we consider a measurable domain $\mathbb A$, a set of probabilities $\cP$ and a probability $Q \notin \cP$ on $\mathbb A$; as well as the set of all $\cP$-e-values, $\cP^\circ$, and the effective null, $\efn{\cP}$ which we now denote by $\cP_\eff$. We always assume $Q \ll \cP$; that is,
\begin{equation}
   (\forall P \in \cP,  P(A) = 0) \implies  Q(A) = 0.
\end{equation}

First, we recall the following main result from \cite{numeraire} concerning \emph{finding} the numeraire.

\begin{proposition}\label[proposition]{prop:fundamental-numeraire}
    Let $f \in \fF(\mathbb A)$. The following are equivalent:
    \begin{itemize}
        \item[(i)] $Q(\{f > 0\}) = 1$, $f \in \cP^\circ$, and for all $g \in \cP^\circ$, $ \Expw_{X\sim Q} \log \frac{g(X)}{f(X)} \le 0$. That is, $f$ is a $Q$-a.s.\ positive $\cP$-e-value that is log-optimal.
        \item[(ii)] $f= \frac{\d Q}{\d P^*}$ for some $P^* \in \cP_\eff$ with $P^* \sim Q$, and $f \in \cP^\circ$. That is, $f$ is a $\cP$-e-value and also a likelihood ratio between $Q$ and an equivalent subprobability in the effective null.
    \end{itemize}
    When this happens, $f$ is the \emph{numeraire e-value} for $\cP$ with respect to $Q$, and $P^*$ the \emph{reverse information projection (RIPr)} of $Q$ onto $\cP$. Additionally, the numeraire $f$ is one of the most e-powerful $\cP$-e-values under $Q$:
    \begin{equation}\label{eqn:most-e-powerful}
        \Expw_{X\sim Q} \log f(X) = \sup_{g \in \cP^\circ} \Expw_{X\sim Q} \log g(X)
    \end{equation}
    where we stipulate $ \Expw_{X\sim Q} \log g(X) = -\infty$ whenever $\Expw_{X\sim Q} (\log g(X))^- = \infty$.
\end{proposition}
As we have already noted before \cref{thm:ripr-lcm},
the characterization of log-optimality (i) above is stronger than the e-power
maximality \eqref{eqn:most-e-powerful}, since (i) is able to distinguish the numeraire $f$ among infinitely e-powerful e-values ($g\in\cP^\circ$ such that $\Expw_{X\sim Q} \log g(X)  = \infty$ ) if they exist.

We now consider two larger nulls.

\begin{definition}
    Let $\cP^2 = \{ P_1 \otimes P_2 : P_1, P_2 \in \cP \} $ and $\cP^2_\iids = \{ P \otimes P: P \in \cP \}$.
\end{definition}

That is, $\cP^2$ is the set of all distributions of independent pairs $(X, Y)$ where both $X$ and $Y$'s distributions belong to $\cP$, whereas $\cP^2_\iids$ is the set of all distributions of \emph{i.i.d.}\ pairs $(X, Y)$ where  $X$ and $Y$'s common distribution belongs to $\cP$.

First, let us discuss the effective nulls $(\cP^2)_\eff$ and $(\cP^2_\iids)_\eff$

\begin{lemma}\label[lemma]{lem:product-eff}
    If $P_1, P_2 \in \cP_\eff$, then $P_1 \otimes P_2 \in (\cP^2)_\eff$. That is, $(\cP_\eff)^2 \subseteq (\cP^2)_\eff$.
\end{lemma}

\begin{proof}
    Let $f$ be a $\cP^2$-e-value. Then, we deduce the following:
    \begin{gather}
      \forall P_3, P_4 \in \cP,\;  \Exp_{(X,Y) \sim P_3 \otimes P_4} f(X, Y) \le 1;
        \\
      \forall P_3, P_4 \in \cP,\;  \Exp_{X \sim P_3} \Exp_{Y \sim P_4} f(X, Y ) \le 1;
        \\
     \forall P_4 \in \cP,\;   x \mapsto \Exp_{Y \sim P_4} f(x, Y) \text{ is a $\cP$-e-value};
         \\
     \forall P_4 \in \cP,\;     \Exp_{X \sim P_1} \Exp_{Y \sim P_4} f(X, Y) \le 1;
        \\
     \forall P_4 \in \cP,\;    \Exp_{Y \sim P_4} \Exp_{X \sim P_1}  f(X, Y) \le 1;
        \\
        y \mapsto \Exp_{X \sim P_1}  f(X, y) \text{ is a $\cP$-e-value};
         \\
       \Exp_{Y \sim P_2} \Exp_{X \sim P_1}  f(X, Y) \le 1;
       \\
        \Exp_{(X,Y) \sim P_1 \otimes P_2} f(X, Y) \le 1.
    \end{gather}
    This concludes the proof.
\end{proof}

\begin{remark}
    If $P \in \cP_\eff$, then it is possible that $P^2 = P\otimes P \notin (\cP^2_\iids)_\eff$.
    That is, $(\cP_\eff)^2_\iids \nsubseteq (\cP^2_\iids)_\eff$.

    Consider the following example. $\mathbb A = \{ 0, 1 \}$. $\cP = \{ \ber(0.1), \ber(0.9) \}$. Then $\ber(0.5) \in \cP_\eff$. However, the two-observation function
    \begin{equation}
        f(X, Y) = \frac{1}{0.18} \id_{\{ X \neq Y \}}
    \end{equation}
    is an e-value for $\ber(0.1) \otimes \ber(0.1)$, for $\ber(0.9)\otimes \ber(0.9)$, but not for $\ber(0.5)\otimes \ber(0.5)$.

    That is, the second i.i.d.\ observation can distinguish $\ber(0.5)$ from these ``extreme'' distributions: it's much more likely for the two observations to differ under $\ber(0.5)$. The previous \cref{lem:product-eff} states that one won't be able to distinguish $\ber(0.5)\otimes  \ber(0.5)$ from the null if the null not only contains $\ber(0.1) \otimes \ber(0.1)$ and $\ber(0.9)\otimes \ber(0.9)$, \emph{but also $\ber(0.1) \otimes \ber(0.9)$ and $\ber(0.1)\otimes \ber(0.9)$}.
\end{remark}

These two types of null thus have different criteria for RIPr tensorization.

\begin{proposition}
    Let $P^*_1 \in \cP_\eff$ be the RIPr of $Q_1$ on $\cP$; $P^*_2 \in \cP_\eff$ be the RIPr of $Q_2$ onto $\cP$. Then, $P^*_1 \otimes P^*_2$ is the RIPr of $Q_1 \otimes Q_2$ on $\cP^2$.
\end{proposition}
\begin{proof}
    First, since $Q_1 \sim P^*_1$ and $Q_2 \sim P^*_2$, $Q_1\otimes Q_2 \sim P^*_1 \otimes P^*_2$. By \cref{lem:product-eff}, we see that $P^*_1 \otimes P^*_2 \in (\cP^2)_\eff$. Thus, by \cref{prop:fundamental-numeraire}, it remains to show that $\frac{\d Q_1\otimes Q_2}{\d P^*_1 \otimes P^*_2}$ is a $\cP^2$-e-value.

    The Radon-Nikodym derivative of product measures is the product of Radon-Nikodym derivatives (see e.g.\ Theorem 8 in \cite{bermudez2025proofs}):
    \begin{equation}
        \frac{\d Q_1\otimes Q_2}{\d P^*_1 \otimes P^*_2}(X, Y) = \frac{\d Q_1}{\d P^*_1 }(X)  \cdot \frac{\d Q_2}{\d P^*_2 }(Y) 
    \end{equation}
    which is a product of two independent e-values, thus again an e-value for $\cP^2$.
\end{proof}

\begin{proposition}
    Let $P^* \in \cP$ (nota bene: not $\cP_\eff$) be the RIPr of $Q$ onto $\cP$. Then, $P^* \otimes P^*$ is the RIPr of $Q \otimes Q$ onto $\cP^2_\iids$.
\end{proposition}

\begin{proof}
    Easy to see that $P^* \otimes P^* \sim Q \otimes Q$, and $P^* \otimes P^* \in \cP^2_\iids \subseteq (\cP^2_\iids)_\eff$. To see $\frac{\d Q\otimes Q}{\d P^* \otimes P^*}$ is a $\cP^2_\iids$-e-value:
    \begin{equation}
        \frac{\d Q\otimes Q}{\d P^* \otimes P^*}(X, Y) = \frac{\d Q}{\d P^* }(X)  \cdot \frac{\d Q}{\d P^* }(Y) 
    \end{equation}
    is a product of two independent e-values, thus again an e-value for $\cP^2_\iids$.
\end{proof}

Sometimes, even though the one-observation RIPr $P^* \notin \cP$, we are able to upper bound it by some probability in $\cP$ and thus still characterize the two-observation RIPr.

\begin{proposition}\label{prop:ripr-bounded-2obs}
    Suppose $P^* \in \cP_\eff$ the RIPr of $Q$ onto $\cP_\eff$ satisfies $P^* \le P$ for some $P \in \cP$. Then, $P^* \otimes P^*$ is the RIPr of $Q \otimes Q$ onto $\cP^2_\iids$.
\end{proposition}

\begin{proof}
    The only thing that needs showing is that $P^* \otimes P^*$ belongs to the effective null $(\cP^2_\iids)_\eff$. To see that, $P^* \otimes P^* \le P \otimes P$, hence whenever $f$ is an e-value for $\cP^2_\iids$, it is an e-value for $P \otimes P$, therefore for $P^* \otimes P^*$.
\end{proof}

We remark that the $P \in \cP$ that upper bounds $P^*$ may not be equivalent to $Q$, and so this does not violate the maximality of $P^*$ demonstrated by \citet[Theorem 4.1]{numeraire}.

Finally, it is straightforward to extend the above to an arbitrary number of i.i.d.\ observations, and thus the optimal e-process.
\begin{theorem}\label{prop:ripr-bounded-nobs}
    Suppose $P^* \in \cP_\eff$ the RIPr of $Q$ onto $\cP_\eff$ satisfies $P^* \le P$ for some $P \in \cP$. Then:
    
    \begin{enumerate}
        \item  ${P^*}^{n}$ is the RIPr of $Q^{n}$ onto $\cP^n_\iids$.
        \item The process
    \begin{equation}
        M^*_n = \prod_{k=1}^n  \frac{\d Q}{\d P^*}(X_k)
    \end{equation}
    is a nonnegative supermartingale under $\cP^{\infty}_\iids$, and it is the log-optimal e-value, e-process, and nonnegative supermartingale for $\cP^{n}_\iids$ at any fixed sample size $n$ under $Q$:
        \begin{equation}
            \Expw_{X_1,\dots, X_n \iid Q} \log \frac{ g_n }{ M^*_n } \le 0
        \end{equation}
        for any $n$-observation e-value (hence any e-process, hence any nonnegative supermartingale) $g_n$ for $\cP^{n}_\iids$.
    \end{enumerate}
   More generally, if the numeraires for $\cP$ w.r.t.\ $Q$, for $\cP^2_\iids$ w.r.t.\ $Q^2$, for $\cP^3_\iids$ w.r.t.\ $Q^3$ ... form an e-process (resp.\ nonnegative supermartingale) for $\cP^\infty_\iids$, it is the log-optimal e-process (resp.\ e-process and nonnegative supermartingale) for $\cP^\infty_\iids$ under $Q^\infty$.

   Further:
   \begin{enumerate}
       \item[3.] At any finite stopping time $\tau$ on the canonical filtration $\sigma(X_1,\dots, X_n)$, $M_\tau^*$ is the numeraire and log-optimal stopped e-process:
        \begin{equation}\label{eqn:log-optimal-stop}
           \Expw_{\{X_n\}\iid Q} \frac{ g_\tau }{ M^*_\tau } \le 1 \quad \text{and}\quad  \Expw_{\{X_n\}\iid Q} \log \frac{ g_\tau }{ M^*_\tau } \le 0
        \end{equation}
        for any e-process $\{ g_n \}$ for $\cP^{n}_\iids$.
   \end{enumerate}
\end{theorem}

\begin{proof}
    Part 1 is a straightforward extension of \cref{prop:ripr-bounded-2obs} above via simply replacing the 2-fold tensorization with $n$-fold. Part 2 follows directly from Part 1. Part 3 generalizes Theorem 12 in \cite{koolen2022log} (allowing stopping times that are not necessarily bounded) and also Theorem 7.11 in \cite{ramdas2024hypothesis} (extending the simple null to composite), for which we note that
    \begin{align}
        & \Expw_{X_1,\dots, X_T \iid Q} \frac{ g_\tau }{ M^*_\tau } 
      \quad  =   \quad \Expw_{X_1,\dots, X_T \iid Q} g_\tau \prod_{k=1}^\tau \frac{\d P^*}{\d Q}(X_k)
        \\
        \le &   \Expw_{X_1,\dots, X_T \iid Q} g_\tau \prod_{k=1}^\tau \frac{\d P}{\d Q}(X_k)
        \quad  \stackrel{(*)}=  \quad  \Expw_{X_1,\dots, X_T \iid P} g_\tau  \le 1.
    \end{align}
    Here, the change-of-measure step $(*)$ is justified the same way as the equality marked $(\star)$ in Proof of Theorem 12 in \cite{koolen2022log} (see \cref{rmk:stopped-change} below). The inequality \eqref{eqn:log-optimal-stop} then follows from Jensen's inequality.
\end{proof}

\begin{remark}\label[remark]{rmk:stopped-change}
\cite{koolen2022log} use an in-line proof for the ``stopped change-of-measure" lemma in the $(*)$ step above, in the slightly nonstandard setting which deals with bounded but possibly ``randomized'' stopping times. In \cref{sec:stopped-change}, we present a
self-contained proof of this separate lemma in the standard filtered space--stopping time setting.  
\end{remark}

\section{Summary}

In this paper, we propose sequential power-one tests for monotonicity and unimodality (for a fixed mode or unknown modes) on $\mathbb Z$ based on e-processes. We also present a complete one-observation picture of these testing problems,
characterizing the set of {all} e-values (thus the set of {all valid tests}) for discrete monotone and $\theta$-unimodal distributions.  On the estimation side, we provide a finite-length confidence set for the mode based on one observation, and an almost surely eventually correct set-valued mode estimator based on an i.i.d.\ sequence of observations.
We partially extend the results to continuous distributions, resonating with some prior work in the shape-constrained inference literature.

\bibliography{main}


\newpage
\appendix

\section{Omitted proofs}\label[appendix]{sec:pf}

\subsection{Proof of \cref{lem:oracle-power1}}\label[appendix]{sec:pf-orcl}

\begin{proof}
    Consider the log-process
    \begin{align}
        & \log  M_n^{(m)} = \sum_{k=1}^n \log E^{  \hat Q_{k-1} ,m}(X_k) \\
        = & \sum_{k=1}^n \log \left\{ 1 + \lambda_{\hat Q_{k-1},m}  ( - \id_{\{ X_k=m \}} + \id_{\{ X_k = m + 1 \}}) \right\}.
    \end{align}
    The conditional mean and variance of these log-summands are
    \begin{align}
       \mu_k := \Exp[ \log E^{  \hat Q_{k-1} ,m}(X_k) | \cF_{k-1} ] = f_Q(m) \log(1 -  \lambda_{\hat Q_{k-1},m}) +   f_Q(m+1) \log(1 +  \lambda_{\hat Q_{k-1},m} )
    \end{align}
    and
    \begin{align}
       v_k :=  & \Var[ \log E^{  \hat Q_{k-1} ,m}(X_k) | \cF_{k-1} ] 
       =  (1-f_Q(m) - f_Q(m+1)) \mu_k^2 + \\ &  f_Q(m)( \log(1 -  \lambda_{\hat Q_{k-1},m}) - \mu_k  )^2 +  f_Q(m+1) (\log(1 +  \lambda_{\hat Q_{k-1},m} )  - \mu_k  )^2 .
    \end{align}
    By the strong law of large numbers, $\lambda_{\hat Q_{k-1},m} $ converges almost surely to $\lambda_{Q,m}$. Therefore, $\mu_k$ converges almost surely to
    \begin{equation}
      \mu := \Expw_{X \sim Q} \log E^{Q,m}(X);
    \end{equation}
   and $ v_k $ converges almost surely to
    \begin{equation}
        v := \Varw_{X \sim Q} \log E^{Q,m}(X) > 0.
    \end{equation}
Now, by the martingale strong law of large numbers,
    \begin{equation}
        \frac{\log  M_n^{(m)}  - \sum_{k=1}^n  \Exp[ \log E^{  \hat Q_{k-1} ,m}(X_k) | \cF_{k-1} ] }{\sum_{k=1}^n   \Var[ \log E^{  \hat Q_{k-1} ,m}(X_k) | \cF_{k-1} ] } \stackrel{a.s.}{\longrightarrow} 0,
    \end{equation}
    implying that
    \begin{equation}
          \frac{\log  M_n^{(m)} }{n} \stackrel{a.s.}{\longrightarrow} \mu \ge  \frac{(f_Q(m+1) - f_Q(m))^2}{2( f_Q(m)+ f_Q(m+1)  )  } > 0
    \end{equation}
     via \cref{lem:log-power-monotone-evalue}.
\end{proof}

\subsection{Proof of \cref{thm:bipolar-monotone}}\label[appendix]{sec:pf-bipolar-M}
The proof of \cref{thm:bipolar-monotone} is made convenient by the following lemma.

\begin{lemma}\label[lemma]{lem:bipolar-seq}
    If $P \in {^\circ\fM}$, and $\{s_n\}$ is an integer-valued sequence satisfying $ s_n \ge n$ for all $n \in \mathbb Z^{\ge 0}$, then $\sum_{n=0}^\infty f_P(s_n) \le 1$.
\end{lemma}
\begin{proof}
    Define the sequence $\{ \rho_n \}$ as $\rho_n = | \{  k : s_k < n \} |$. First, since $\{  k : s_k < n \} \subseteq \{ 0,\dots, n-1 \}$ we see that $\rho_n \le n$. Second, $0 = \rho_0 \le \rho_1 \le \dots$. Third,
    \begin{equation}
        \sum_{n=0}^\infty f_P(n)(\rho_{n+1} - \rho_n) =   \sum_{n=0}^\infty f_P(n)| \{  k : s_k=n \} | = \sum_{n=0}^\infty f_P(s_n).
    \end{equation}
    Finally, noting that $P \in {^\circ\fM}$ and $n \mapsto \rho_{n+1} - \rho_n \in \fM$, we have the above $\le 1$.
\end{proof}

\begin{proof}[Proof of \cref{thm:bipolar-monotone}]
    Let $P \in {^\circ\fM}$. First, since the constant function $E(n) = 1$ belongs to $\fM$, we see that $\sum f_P(n) = \sum f_P(n) E(n) \le 1$, i.e., $P$ is a subprobability. Therefore, for any $n$, the supremum $\sup_{k \ge n} f_P(k)$ is in fact a maximum, and $\argmax_{k \ge n} f_P(k)$ always exists. 
    
    For each $n$, we let $s_n = \argmax_{k \ge n} f_P(k)$, tie broken arbitrarily.
    Then, since $s_n \ge n$, by \cref{lem:bipolar-seq} we have
    \begin{equation}
       1 \ge  \sum_{n=0}^\infty f_P(s_n) = \sum_{n=0}^\infty \left( \sup_{k \ge n} f_P(k) \right).
    \end{equation}
    Therefore, $P$ is upper bounded by the monotone decreasing subprobability $P'$ defined as $f_{P'}(n) = \sup_{k \ge n} f_P(k)$, concluding that $P \in \cM^-$ and consequently that ${^\circ\fM} \subseteq \cM^-$. The inclusion $\cM^- \subseteq {^\circ\fM}$ is straightforward.
\end{proof}

\subsection{Proof of \cref{prop:xqx-bipolar}}\label[appendix]{sec:pf-xqx-bipolar}
\begin{proof}
    To see $\cM' \subseteq {^\circ\fQ}$, suppose $P \in \cM'$. For any $Q \in \cP(\mathbb Z^{\ge 0})$,
\begin{equation}\label{eqn:step1}
      \int (k+1) f_Q(k) P(\d k)  = \sum_{k=0}^\infty (k+1) f_P (k)  f_Q (k) \le  \sum_{k=0}^\infty  f_Q (k) \le 1,
\end{equation}
concluding that $\cM' \subseteq {^\circ \fQ}$. 

To see ${^\circ\fQ} \subseteq \cM'$,  suppose $P \in {^\circ \fQ}$. Then, for any $n \in \mathbb Z^{ \ge 0}$, we take $Q$ to be the point mass at $n$. Since $P \in {^\circ \fQ}$ and $(n+1)f_Q(k)  \in \fQ$,
    \begin{equation}
        1 \ge   \int (k+1) f_Q(k) P(\d k) = f_P(n)(n+1),
    \end{equation}
    implying that $P \in \cM'$. This implies that ${^\circ \fQ} \subseteq \cM'$,  and also $\fQ \subseteq ({^\circ \fQ})^\circ = \cM'^\circ$. 
    
    Finally, let us show $\cM'^\circ \subseteq \fQ$. Suppose $E \in \cM'^\circ$. Consider the ``maximal'' measure $P \in \cM'$ with $f_P(n) = 1/(n+1)$. Since $E \in \cM'^\circ$,
      \begin{equation}
        1 \ge   \int E(k) P(\d k) = \sum_{k=0}^\infty E(k)/(k+1),
    \end{equation}
    concluding that $E(n)/(n+1)$ is the measure mass function of some subprobability. Therefore, $E \in \fQ$.
\end{proof}

\subsection{Proof of \cref{thm:polar-unimodal-Z}}\label[appendix]{sec:pf-polar-Z}
Recall that the function class we wanted to show to equal $\fD_0$ (taking $\theta = 0$ without loss of generality) is
\begin{equation}\label{eqn:D0}
        \fD_{0}' = \left\{ n \mapsto \begin{cases}
            \rho_{n + 1} -\rho_{n}, & n > 0 \\
            \rho_1 + \eta_1 - 1, &  n = 0 \\
            \eta_{ - n + 1} - \eta_{ - n}, & n < 0
        \end{cases} \;  : \;  
        \begin{matrix}
\rho_1 + \eta_1 \ge 1; \ \rho_1 \le \dots \le  \rho_n \le n \\ \text{and }
 \eta_1 \le \dots \le  \eta_n \le n
       \text{ for all } n 
\end{matrix}
 \right\}.
    \end{equation}
Let us first show that there is an equivalent way to parametrize this function class.
\begin{lemma}\label[lemma]{lem:shifting-sequences}
    Let
    \begin{equation}\label{eqn:D0prime}
        \fD_{0}'' = \left\{ n \mapsto \begin{cases}
            \rho_{n + 1} -\rho_{n}, & n > 0 \\
            \rho_1 + \eta_1 - 1, &  n = 0 \\
            \eta_{ - n + 1} - \eta_{ - n}, & n < 0
        \end{cases} \;  : \;  
        \begin{matrix}
\rho_1 = \eta_1 \in [1/2, 1]; \ \rho_1 \le \dots \le  \rho_n  \\ \text{and }
 \eta_1 \le \dots \le  \eta_n
       \text{ for all } n; \\
       \rho_n + \eta_m \le n +m \text{ for all } n,m
\end{matrix}
 \right\}.
    \end{equation}
Then $\fD_{0}' = \fD_{0}''$.
\end{lemma}

\begin{proof} Given two sequences $( \rho_n ), (\eta_n )$, we say $E(n)= \begin{cases}
            \rho_{n + 1} -\rho_{n}, & n > 0 \\
            \rho_1 + \eta_1 - 1, &  n = 0 \\
            \eta_{ - n + 1} - \eta_{ - n}, & n < 0
        \end{cases}$ is the function they generate.

    To see $\fD_{0}' \subseteq \fD_{0}''$, if a function is represented as \eqref{eqn:D0}, define
    \begin{equation}
        \rho_{n}' = \rho_n + \frac{\eta_1 - \rho_1}{2}, \quad \eta_{n}' = \eta_n + \frac{\rho_1 - \eta_1}{2}
    \end{equation}
    and the sequences $( \rho_n' ), ( \eta_n' )$ generate the same function satisfying \eqref{eqn:D0prime}.

    To see $\fD_{0}'' \subseteq \fD_{0}'$, if a function is represented as \eqref{eqn:D0prime},  consider $\varepsilon = \max\{\sup_n(\rho_n - n), \sup_n(\eta_n - n) \}$. If $\varepsilon \le 0$, then the same sequences satisfy the conditions in \eqref{eqn:D0} and generate the same function. If $\varepsilon > 0$, without loss of generality let us assume $\sup_n(\rho_n - n) = \varepsilon$. First,
    note that $\rho_n + 1/2 \le  \rho_n + \eta_1 \le n + 1$ so
    $\varepsilon \le 1/2$. Second, since $\rho_n + \eta_m \le n+m$, we see that $\sup_m(\eta_m - m) \le \inf_n(n - \rho_n) = -\varepsilon$. We can define $\eta_n' = \eta_n + \varepsilon$ and $\rho_n' = \rho_n - \varepsilon$. Then, $0 \le \rho_n' \le n$, $0 \le \eta_n' \le n$ and the sequences $( \rho_n' ), ( \eta_n' )$ generate the same function satisfying \eqref{eqn:D0}.
\end{proof}

\begin{proof}[Proof of \cref{thm:polar-unimodal-Z}] Since the problem is translation invariant, it suffices to consider the case $\theta = 0$. First, suppose $E \in \fD_0'$ with 
\begin{equation}
    E(n) =  \begin{cases}
            \rho_{n + 1} -\rho_{n}, & n > 0 \\
            \rho_1 + \eta_1 - 1, &  n = 0 \\
            \eta_{ -n + 1} - \eta_{ -n}, & n < 0
        \end{cases}
\end{equation}
and consider any $P \in \cD_0$.
\begin{align}
   & \sum_{n= -N+1}^{N-1} E(n) f_P(n) =  \sum_{n=1}^{N-1} E(n) f_P(n) +  \sum_{n=1}^{N-1} E(-n) f_P(-n) + E(0)f_P(0)
   \\
   & \text{(in analogy to Proof of \cref{thm:polar-of-M})}
   \\
   \le & \sum_{n=1}^{N-1} f_P(n) + (\rho_N - N)f_P(N) - (\rho_1 - 1)f_P(1) +
   \\  &   \sum_{n=1}^{N-1} f_P(-n) + (\eta_N - N)f_P(-N) - (\eta_1 -  1)f_P(-1) +  
   \\  &   (\rho_1 + \eta_1 - 1)f_P(0) 
   \\
   \le & \sum_{n=-N+1}^N f_P(n)  -(\rho_1 - 1)f_P(1) - (\eta_1 - 1)f_P(-1) + (\rho_1 + \eta_1 - 2) f_P(0) \\
   \le & 1 - (\rho_1 - 1)f_P(0) - (\eta_1 - 1)f_P(0) + (\rho_1 + \eta_1 - 2) f_P(0)  = 1.
\end{align}
So $\Expw_{X \sim P}(E(X)) = \lim_{N \to \infty}\sum_{n= -N+1}^{N-1} E(n) f_P(n)  \le 1$. Therefore, $E \in \cD_0^\circ$. This shows that $\fD_0' \subseteq \cD_0^\circ = \fD$.

Second, suppose $E \in \fD_0 = \cD_0^\circ$. Define the increasing sequences
    \begin{equation}
        \rho_{n} = \frac{1+E(0)}{2} + \sum_{k=1}^{n-1} E(k), \quad \eta_n = \frac{1+E(0)}{2} + \sum_{k=1}^{n-1} E(-k) 
    \end{equation}
    which satisfy,
\begin{equation}
    E(n) =  \begin{cases}
            \rho_{n + 1} -\rho_{n}, & n > 0 \\
            \rho_1 + \eta_1 - 1, &  n = 0 \\
            \eta_{ -n + 1} - \eta_{ -n}, & n < 0
        \end{cases}
\end{equation}
    Now, for any $n,m \in \mathbb Z^{\ge 1}$, consider $\operatorname{unif}_{\{ -(m-1),\dots, n-1 \}} \in \cD_0$. We see that $\rho_n + \eta_m \le n + m$. Also, $1/2 \le \rho_1 = \eta_1 \le 1$. By \cref{lem:shifting-sequences}, we see that $E\in \fD_0'' = \fD_0'$. Therefore $\fD_0 = \cD_0^\circ \subseteq \fD_0'$. 
\end{proof}

\subsection{Proof of \cref{thm:bipolar-unimodal}}\label[appendix]{sec:pf-bipolar-D}

\begin{lemma}\label[lemma]{lem:bipolar-seq-Z}
    If $P \in {^\circ\fD_0}$, $(s_n)_{n \ge 1}$, $(t_n)_{n \ge 1}$ are integer-valued sequences satisfying $s_n \ge n$, $t_n \le -n$ for all $n \in \mathbb Z^{\ge 1}$, and arbitrary $u \in \mathbb Z$, then $f_P(u)+ \sum_{n=1}^\infty  (f_P(s_n) + f_P(t_n)) \le 1$.
\end{lemma}
\begin{proof}
We further define $s_0 = t_0 = u$, and then
    define the sequence $(\rho_n )$ as 
    $\rho_n = | \{  k \in \mathbb Z^{\ge 0} : s_k < n \} |$ and $( \eta_n )$ as $\eta_n = | \{  k \in \mathbb Z^{\ge 0} : t_k > - n \} |$. Note that $\rho_n \le n$, $\eta_n \le n$, and that $\rho_1 = \id_{\{ u \le 0\} }$, $\eta_1 = \id_{\{ u \ge 0\} }$, and that $( \rho_n )$, $( \eta_n )$ are both increasing with $\rho_1 + \eta_1 \ge 1$. Now
    note that, since $P \in {^\circ\fD_0}$ and these two sequences $( \rho_n )$, $(\eta_n)$ generate a function in $\fD_0$:
    \begin{align}
       & 1 \ge  f_P(0) (\rho_1 + \eta_1 - 1) + 
       \sum_{n=1}^\infty f_P(n) (\rho_{n+1} - \rho_n) +  \sum_{n=1}^\infty f_P(-n) (\eta_{n+1} - \eta_n)
       \\ 
       = & f_P(0) (\rho_1 + \eta_1 - 1) + \sum_{n=1}^\infty f_P(n) |\{ k: s_k = n \}| +  \sum_{n=1}^\infty f_P(-n) |\{ k: t_k = -n \}|
       \\
       = & f_P(0) (\rho_1 + \eta_1 - 1) +  \left(\sum_{n=1}^\infty f_P(s_n) + \id_{\{ u > 0 \}} f_P(u) \right)+  \left(\sum_{n=1}^\infty f_P(t_n)  + \id_{\{ u < 0 \}} f_P(u) \right)
       \\
       = & f_P(0) (\id_{\{ u \ge 0 \}} +\id_{\{ u \le 0 \}} - 1) +   \id_{\{ u \neq 0 \}} f_P(u)+\sum_{n=1}^\infty (f_P(s_n) +  f_P(t_n))
       \\
       = & f_P(0) \id_{\{ u = 0 \}} +   \id_{\{ u \neq 0 \}} f_P(u)+\sum_{n=1}^\infty (f_P(s_n) +  f_P(t_n))
        \\
       = & \id_{\{ u = 0 \}} f_P(u)  +   \id_{\{ u \neq 0 \}} f_P(u)+\sum_{n=1}^\infty (f_P(s_n) +  f_P(t_n))
       \\
        = & f_P(u)+\sum_{n=1}^\infty (f_P(s_n) +  f_P(t_n)),
    \end{align}
    concluding the proof.
\end{proof}

\begin{proof}[Proof of \cref{thm:bipolar-unimodal}]  Again it suffices to consider $\theta=0$. Let $P \in {^\circ\fD_{0}}$. By considering the constant function $E(n) = 1 \in \fD_0$ we again see that $P$ must be a subprobability. Therefore, for any $n \ge 1$ the suprema
\begin{equation}
    \sup_{k \ge n} f_P(k), \quad \sup_{k \le -n} f_P(k), \quad  \sup_{k \in \mathbb Z} f_P(k)
\end{equation}
are all maxima. So for each $n \ge 1$ we can pick an $s_n \ge n$ that satisfies
    \begin{equation}
        f_P(s_n) = \sup_{k \ge n} f_P(k)
    \end{equation}
    and a $t_n \le -n$ that satisfies
    \begin{equation}
          f_P(-t_n) =  \sup_{k \le -n} f_P(k);
    \end{equation}
    We can also pick $u \in \mathbb Z$ that satisfies
    \begin{equation}
        f_P(u) = \sup_{k \in \mathbb Z} f_P(k).
    \end{equation}
    Therefore, if we define the measure $P'$ as
    \begin{equation}
        f_{P'}(n) = \begin{cases}
            \sup_{k \ge n} f_{P}(k) & n \ge 1 \\
            \sup_{k \in \mathbb Z} f_{P}(k) & n \ge 0 \\
             \sup_{k \le n} f_{P}(k) & n \le -1 
        \end{cases}
    \end{equation}
    we have
    \begin{equation}
        \sum_{n \in \mathbb Z}  f_{P'}(n) \le f_P(u)+ \sum_{n=1}^\infty(f_P(s_n) + f_P(t_n)) \stackrel{(*)}\le 1 
    \end{equation}
    for any $\varepsilon$, where the inequality $(*)$ is due to \cref{lem:bipolar-seq-Z}. Therefore, $P'$ is a subprobability. Since $P'$ upper bounds $P$ and is $0$-unimodal, we see that $P \in \cD_{0}^-$.
\end{proof}

\subsection{Proof of \cref{cor:1obs-CI}}\label[appendix]{sec:pfci}

\begin{proof}
   Let us first prove $  \CI(X;\alpha, 0)$ is a valid $(1-\alpha)$-CI for $\theta$.
    For each $\theta > T_\alpha - 1$, let $q_\theta$ be the PMF of the uniform distribution on
    \begin{equation}
      I_\theta :=  \left\{ \left[ -\frac{\theta}{T_\alpha-1} , 0\right) \cup \left(0, \frac{\theta}{T_\alpha+1} \right] \right\} \cap \mathbb Z.
    \end{equation}
    It is clear that $q_\theta$'s support $I_\theta$ is a subset of $\mathbb Z^{\le \theta}$. Further, the cardinality of $I_\theta$ satisfies
    \begin{equation}
      0 <  |I_\theta| \le \frac{\theta}{T_\alpha-1} + \frac{\theta}{T_\alpha+1}  = \frac{ 2T_\alpha \theta }{T_\alpha^2-1}.
    \end{equation}
    Therefore, for each $n \in I_\theta$, the realized e-value $E(n;\theta)$ satisfies
    \begin{equation}
        E(n;\theta) = |  (\theta - n) + 1 | \cdot | I_\theta|^{-1} > \frac{T_\alpha \theta}{T_\alpha + 1} \cdot \frac{T_\alpha^2-1}{ 2T_\alpha \theta } = 1/\alpha.
    \end{equation}
    Similarly, for each $\theta < - (T_\alpha - 1)$, let $q_\theta$ be the PMF of the uniform distribution on
     \begin{equation}
      I_\theta :=  \left\{ \left[  \frac{\theta}{T_\alpha+1} , 0\right) \cup \left(0, -\frac{\theta}{T_\alpha-1} \right] \right\} \cap \mathbb Z.
    \end{equation}
    Then for each $n \in I_\theta$ we also have $ E(n;\theta)  > 1/\alpha$. Finally, for each $|\theta| \le T_\alpha - 1$ we simply take $q_\theta$ to be the point mass on $\theta$ (in this case for all $n$, $E(n;\theta) \le 1$).

    From above, we see that the confidence set $ \CI =   \{ \theta \in \mathbb Z :    E(X;\theta) < 1/\alpha  \}$ from \cref{cor:test-Z} reads
    \begin{equation}
        \CI = \{ \theta: X \notin I_{\theta} \}  = \begin{cases}
              ( X - T_\alpha |X|, X + T_\alpha |X| ) \cap \mathbb Z , & X \neq 0; \\
              \mathbb Z , & X = 0.
        \end{cases}
    \end{equation}
    This is exactly $\CI(X; \alpha, 0)$. For $\phi \neq 0$, the fact that $\CI(X; \alpha, \phi)$ is still a valid CI can be seen by a simple shifting argument, and the final argument is a simple union bound.
\end{proof}

\subsection{Proof and an application of \cref{thm:cont}}\label[appendix]{sec:pf-cont}
\begin{proof} 
For any $a, w > 0$, since $L_{a,a+w}(X)$ is an e-value for $\cV$ due to \cref{lem:bump}, so is
\begin{equation}
  E(X;w) =  \int_{0}^\infty L_{a, a + w}(x) q(a) \d a
\end{equation}
due to Fubini's theorem. Note that
\begin{align} 
    & E(X;w)
    \\
    = & \int_{0}^\infty  \frac{a+w}{w} \id_{\{  a < X \le a + w \}} q(a) \d a
    \\
    = &  \frac{1}{w}  \int_{(X-w)^+}^{X^+} a q(a) \d a +  \int_{(X-w)^+}^{X^+}   q(a) \d a.
\end{align}
Consider the e-values $E(X; 1/n)$ for $n=1,2,\dots$. First,
\begin{equation}
    \lim_{n\to\infty}  E(X; 1/n) =  X^+ q(X^+)
\end{equation}
due to the left continuity of $q$ at $X^+$.
Second, suppose $\sup_{x \ge 0} (x+1) q(x) = M$. We have,
\begin{equation}
    |E(X;1/n)| \le  \max_{x \in [(X-w)^+, X^+]} (xq(x) + q(x)) \le M.
\end{equation}
The dominated convergence theorem then shows that the limit $ \lim_{n\to\infty}  E(X; 1/n) =  X^+ q(X^+)$ is an e-value for $\cV$ as well.
\end{proof}

An application of \cref{thm:cont} is the following confidence interval for the mode $\theta$. First, we establish the following straightforward analog to \cref{cor:test-Z}.

\begin{proposition}\label[proposition]{cor:test-R}
    Let $\{q_\theta : \theta \in \mathbb R \}$ be a family of continuous PDFs on $\mathbb R$ such that each $q_\theta$ is supported either on $\mathbb R^{\ge \theta}$ or $\mathbb R^{\le \theta}$ with bounded $(|x-\theta| + 1)q(x)$.
    Let \begin{gather}
        E(X;\theta) = |X-\theta| \cdot q_\theta(X),
    \end{gather}
    Then each $ E(X;\theta)$ is an e-value for $\cU_{\theta}$. 
    Consequently,
    \begin{equation}\label{eqn:1obs-ci}
     \CI =   \{ \theta \in \mathbb Z :    E(X;\theta) < 1/\alpha  \}
    \end{equation}
    is a $(1-\alpha)$-confidence interval for $\theta$.
\end{proposition}

By choosing $q_\theta$ to be the continuous uniform distribution over a vicinity of 0, we generalize \cref{cor:1obs-CI}.

\begin{corollary}\label[corollary]{cor:1obs-CI-R}
     Define $T_\alpha = 2/\alpha + 1$. For any $\phi \in \mathbb R$,  the set
     \begin{equation}
              ( X - T_\alpha |X - \phi|, X + T_\alpha |X -  \phi| )
    \end{equation}
     is a $(1-\alpha)$-confidence interval for the mode $\theta$. 
\end{corollary}
\begin{proof}
   It suffices to prove the $\phi = 0$ case.
    For each $\theta > 0$, let $q_\theta$ be the PDF of the uniform distribution on
    \begin{equation}
      I_\theta :=   \left[ -\frac{\theta}{T_\alpha-1} ,  \frac{\theta}{T_\alpha+1} \right].
    \end{equation}
    It is clear that $q_\theta$'s support $I_\theta$ is a subset of $\mathbb R^{\le \theta}$. Further, the length of $I_\theta$ satisfies
    \begin{equation}
       |I_\theta| \le \frac{\theta}{T_\alpha-1} + \frac{\theta}{T_\alpha+1}  = \frac{ 2T_\alpha \theta }{T_\alpha^2-1}.
    \end{equation}
    Therefore, for each $x \in I_\theta$, the realized e-value $E(x;\theta)$ satisfies
    \begin{equation}
        E(n;\theta) = |  \theta - x | \cdot | I_\theta|^{-1} \ge \frac{T_\alpha \theta}{T_\alpha + 1} \cdot \frac{T_\alpha^2-1}{ 2T_\alpha \theta } = 1/\alpha.
    \end{equation}
    Similarly, for each $\theta < 0$, let $q_\theta$ be the PMF of the uniform distribution on
     \begin{equation}
      I_\theta :=  \left[  \frac{\theta}{T_\alpha+1} , -\frac{\theta}{T_\alpha-1} \right] .
    \end{equation}
    Then for each $x \in I_\theta$ we also have $ E(x;\theta)  > 1/\alpha$. Finally, if $\theta = 0$ we simply take $q_\theta$ to be the uniform distribution on $[\theta, \theta + 1]$ (in this case for all $x$, $E(x;\theta) \le 1$).

    From above, we see that the confidence set $ \CI =   \{ \theta \in \mathbb Z :    E(X;\theta) < 1/\alpha  \}$ from \cref{cor:test-R} reads
    \begin{equation}
        \CI = \{ \theta: x \notin I_{\theta} \}  =   ( X - T_\alpha |X |, X + T_\alpha |X | ),
    \end{equation}
    concluding the proof.
\end{proof}

We finally note that this CI is longer than the CI by \cite{Edelman01111990}, which we quote as \cref{prop:edelci}.

\subsection{Stopped change of measure lemma} \label[appendix]{sec:stopped-change}
Let $P$ and $Q$ be probability measures defined on $(\Omega, \mathcal{F})$. For each $n \in \mathbb{N}$, let $P_n$ and $Q_n$ denote the restrictions of $P$ and $Q$ to the $\sigma$-algebra $\mathcal{F}_n$. We assume that $P_n$ is absolutely continuous with respect to $Q_n$, denoted as $P_n \ll Q_n$. The Radon-Nikodym derivative at time $n$ is given by the random variable $L_n$:
\begin{equation}
    L_n = \frac{\d P_n}{\d Q_n}.
\end{equation}
In a standard sequential setting with independent observations $X_1, X_2, \dots$, this takes the form of the likelihood ratio $L_n = \prod_{k=1}^n \frac{\d P}{\d Q}(X_k)$.
\begin{lemma}[Stopped change of measure]
Let $\tau$ be an $(\mathcal{F}_n)$-stopping time such that $P(\tau < \infty) = 1$ and $Q(\tau < \infty) = 1$. Let $(g_n)_{n \in \mathbb{N}}$ be an $(\mathcal{F}_n)$-adapted stochastic process. If either $g_n \ge 0$ almost surely for all $n$, or $\Exp_P[|g_\tau|] < \infty$, then
\begin{equation}
    \Exp_Q[g_\tau L_\tau] = \Exp_P[g_\tau].
\end{equation}
\end{lemma}

\begin{proof}
First, consider the case where the process is nonnegative, meaning $g_n \ge 0$ for all $n$. We evaluate the expectation under $Q$ by partitioning the event $\{\tau < \infty\}$ into the disjoint union of events $\{\tau = n\}$ for $n \ge 1$:
\begin{equation}
    \Exp_Q[g_\tau L_\tau] = \Exp_Q\left[ g_\tau L_\tau \sum_{n=1}^\infty \id_{\{\tau = n\}} \right] = \Exp_Q\left[ \sum_{n=1}^\infty g_n L_n \id_{\{\tau = n\}} \right].
\end{equation}
Since all terms in the sum are nonnegative, Tonelli's theorem (or the monotone convergence theorem) allows us to interchange the expectation and the infinite series:
\begin{equation}
    \Exp_Q\left[ \sum_{n=1}^\infty g_n L_n \id_{\{\tau = n\}} \right] = \sum_{n=1}^\infty \Exp_Q\left[ g_n L_n \id_{\{\tau = n\}} \right].
\end{equation}
By the definition of a stopping time, the event $\{\tau = n\}$ is $\mathcal{F}_n$-measurable. Since the process $(g_n)$ is adapted, $g_n$ is also $\mathcal{F}_n$-measurable. Consequently, the random variable $X_n = g_n \id_{\{\tau = n\}}$ is $\mathcal{F}_n$-measurable. 

The definition of the Radon-Nikodym derivative $L_n = \d P_n / \d Q_n$ on $\mathcal{F}_n$ implies that for any non-negative $\mathcal{F}_n$-measurable random variable $X_n$:
\begin{equation}
    \Exp_Q[X_n L_n] = \Exp_P[X_n].
\end{equation}
Applying this property to $X_n = g_n \id_{\{\tau = n\}}$, we obtain:
\begin{equation}
    \Exp_Q\left[ g_n L_n \id_{\{\tau = n\}} \right] = \Exp_P\left[ g_n \id_{\{\tau = n\}} \right].
\end{equation}
Substituting this back into our infinite series gives:
\begin{equation}
    \Exp_Q[g_\tau L_\tau] = \sum_{n=1}^\infty \Exp_P\left[ g_n \id_{\{\tau = n\}} \right].
\end{equation}
Applying Tonelli's theorem once more under the measure $P$ (valid due to nonnegativity), we bring the sum back inside the expectation:
\begin{equation}
    \sum_{n=1}^\infty \Exp_P\left[ g_n \id_{\{\tau = n\}} \right] = \Exp_P\left[ \sum_{n=1}^\infty g_n \id_{\{\tau = n\}} \right] = \Exp_P[g_\tau \id_{\{\tau < \infty\}}].
\end{equation}
Since we assumed that $\tau$ is finite $P$-almost surely (i.e., $P(\tau < \infty) = 1$), it follows that $\id_{\{\tau < \infty\}} = 1$ $P$-a.s., which yields $\Exp_P[g_\tau \id_{\{\tau < \infty\}}] = \Exp_P[g_\tau]$. This completes the proof for the nonnegative case.

For the general case where $g$ takes both positive and negative values, we assume that $\Exp_P[|g_\tau|] < \infty$. We can decompose the adapted process into its positive and negative parts: $g_n = g_n^+ - g_n^-$, where $g_n^+ = \max(g_n, 0)$ and $g_n^- = \max(-g_n, 0)$. Both $g^+$ and $g^-$ are non-negative adapted processes. Applying the result proven above to $|g_\tau| = g_\tau^+ + g_\tau^-$, we have:
\begin{equation}
    \Exp_Q[|g_\tau| L_\tau] = \Exp_P[|g_\tau|] < \infty,
\end{equation}
which implies that $g_\tau L_\tau$ is integrable with respect to $Q$. Applying the nonnegative result to $g^+$ and $g^-$ individually and using the linearity of expectation yields:
\begin{align*}
    \Exp_Q[g_\tau L_\tau] &= \Exp_Q[g_\tau^+ L_\tau] - \Exp_Q[g_\tau^- L_\tau] \\
    &= \Exp_P[g_\tau^+] - \Exp_P[g_\tau^-] \\
    &= \Exp_P[g_\tau].
\end{align*}
This concludes the proof.
\end{proof}

\end{document}